\documentclass[12pt]{amsart}
\usepackage[a4paper,margin=30mm]{geometry}
\usepackage[all]{xy}
\usepackage{tikz}
\usepackage{rotating}
\usepackage{xspace}
\usepackage[usenames,dvipsnames]{pstricks}

\usepackage{eucal}

\usepackage{amssymb,
amsmath}
\usepackage[usenames,dvipsnames]{pstricks}
\usepackage{epsfig}
\usepackage{pst-grad} 
\usepackage{pst-plot} 
\usepackage{tikz-cd}

\usetikzlibrary{arrows.meta}

\tikzcdset{
arrow style=tikz,
diagrams={>={Classical TikZ Rightarrow[length=2.0mm, width=1.25mm]}}
}

\usepackage[letterspace=-45]{microtype}

\usepackage[raiselinks=false,colorlinks=true,citecolor=blue,urlcolor=blue,linkcolor=blue,bookmarksopen=true]{hyperref}

\usepackage{ifthen,changepage}
\usepackage{xargs}

\usepackage[shortlabels]{enumitem}

\catcode`\@=11
\def\@evenfoot{\rule{0pt}{20pt}[\today] \hfill [{\tt \jobname.tex}]}
\def\@oddfoot{\rule{0pt}{20pt}{[\tt \jobname.tex}]\hfill [\today]}
\catcode`\@=13

\newtheorem{theorem}{Theorem}[section]
\newtheorem{proposition}[theorem]{Proposition}
\newtheorem{lemma}[theorem]{Lemma}

\newtheorem{corollary}[theorem]{Corollary}

\theoremstyle{definition}
\newtheorem{definition}[theorem]{Definition}
\newtheorem{example}[theorem]{Example}

\newtheorem{remark}[theorem]{Remark}

\newtheorem{taller}[theorem]{$\!\!$}

\newenvironment{blanko}[1]%
{\begin{taller}{\normalfont\bfseries  #1}\normalfont}%
{\end{taller}}

\makeatletter
\providecommand\@dotsep{5}
\makeatother

\newcommand{\commutes}{\text{\texttt{"}}}

\def\Ob{{\rm Ob}}

\def\catO{{\mathbb O}}
\def\catP{{\mathbb P}}

\usepackage[rm={oldstyle=false}, tt={monowidth=true}]{cfr-lm}
\newcommand{\DDD}{\textnormal{\textsf{\fontseries{sbc}\selectfont{D}\fontseries{n}\selectfont}}}

\newcommand{\SSS}{\textnormal{\textsf{\fontseries{sbc}\selectfont{S}\fontseries{n}\selectfont}}}

\newcommand{\strictnerve}{\mathrm{N}}
\newcommand{\ltnerve}{\mathrm{N}\kern-0.5pt_{\ell \kern-0.4pt t}}

\newcommand{\nerexs}{\ltnerve}

\def\nop{\strictnerve\kern-1pt_{opd}}

\providecommand{\kat}[1]{\textnormal{\lsstyle{{\texttt{#1}}}}}

\newcommand{\Grpd}{\kat{Grpd}}

\newcommand{\Set}{\kat{Set}}
\newcommand{\Cat}{\kat{Cat}}
\newcommand{\Catlt}{\Cat_{\ell \kern-0.4pt t}}
\newcommand{\OpCat}{\kat{OpCat}}

\newcommand{\SMGp}{\kat{SMGp}}
\newcommand{\PrSh}{\kat{Pr}}

\newcommand{\X}{\mathrm{X}}
\newcommand{\Y}{\mathrm{Y}}
\newcommand{\A}{\mathrm{A}}
\newcommand{\E}{\mathrm{E}}
\newcommand{\B}{\mathrm{B}}

\renewcommand{\P}{\mathrm{P}}

\newcommand{\un}{\relax}

\DeclareMathAlphabet{\mathbbe}{U}{bbold}{m}{n}
\newcommand{\simplexcategory}{\mathbf{\Delta}}
\newcommand{\simplexcategoryop}{\mathbf{\Delta}\kern-1.2pt\op}

\newcommand{\tEM}{\simplexcategory\kern-1pt^\ttt}

\newcommand{\tKLEISLI}{\simplexcategory_\ttt}

\newcommand{\ttt}{\mathsf{t}}

\newcommand{\uuu}{\mathsf{u}}

\newcommand{\kkk}{\mathsf{k}}

\newcommand{\iii}{\mathsf{i}}

\def\ikeo{{IKEO}\xspace}

\DeclareRobustCommand\upperstar{%
  \mathchoice%
    {\kern0pt\raise0.55ex\hbox{$\displaystyle *$}\kern0.8pt}
    {\kern0pt\raise0.58ex\hbox{$\textstyle *$}\kern0.8pt}
    {\kern0pt\raise0.45ex\hbox{$\scriptstyle *$}\kern0.4pt}
    {\kern0pt\raise0.4ex\hbox{$\scriptscriptstyle *$}\kern0.2pt}
}%

\def\Fin{{\mathbb{F}\mathrm{in}}}
\def\Fin{{\mathbb{F}}}

\def\calP{{\mathcal P}}

\def\dd{\mathsf{d}}
\def\ss{\mathsf{s}}
\def\cc{\mathsf{c}}
\def\cc{c}

\makeatletter
\newcommand{\xRightarrow}[2][]{\ext@arrow 0359\Rightarrowfill@{#1}{#2}}
\makeatother

\makeatletter
\newcommand{\xLeftarrow}[2][]{\ext@arrow 0359\Leftarrowfill@{#1}{#2}}
\makeatother

\newcommand{\xto}{\xrightarrow}

\newcommand{\drpullback}{\arrow[phantom]{dr}[very near start,description]{\lrcorner}}

\usepackage[all]{xy}
\newcommand{\cd}[2][]{\vcenter{\hbox{\xymatrix#1{#2}}}}

\newcommand{\name}[1]{\ulcorner #1\urcorner}
\newcommand{\isopil}{\stackrel{\raisebox{0.1ex}[0ex][0ex]{\(\sim\)}}%
			{\raisebox{-0.15ex}[0.28ex]{\(\rightarrow\)}}}
\newcommand{\isleftadjointto}{\dashv}

\newcommand{\thg}{\,{\mathord{\text{--}}}\,}

\newcommand{\fib}{\varphi}

\def\fnerve{\P}

\def\and{{\mbox { and }}}
\def\listtodoname{List of Todos}
\def\listoftodos{\@starttoc{tdo}\listtodoname}
\makeatother

\newcommand{\Aut}{\operatorname{Aut}}

\providecommand{\norm}[1]{\left| {#1}\right|}

\newcommand{\CC}{\mathcal{C}}
\newcommand{\EE}{\mathcal{E}}
\newcommand{\finset}{\mathbb{F}}
\newcommand{\op}{^{\text{{\rm{op}}}}}
\newcommand{\id}{\operatorname{id}}

\title{Operadic categories as (pseudo)-simplicial groupoids}
\author[M.~Batanin, J.~Kock, M.~Weber]{Michael Batanin, Joachim Kock,  Mark Weber}


\begin{document}

\maketitle

\baselineskip 16pt plus 2pt minus 1pt

\setcounter{secnumdepth}{3}
\setcounter{tocdepth}{1}

\begin{abstract}
  From any operadic category $\catO$ we construct a simplicial groupoid
  $\X$ (slightly pseudo in a specific way), called the operadic nerve. It
  integrates all the structure of chosen-local-terminals, fibre functor,
  and cardinality functor into a single simplicial groupoid, which can be
  seen as an undecking of the ordinary nerve of $\catO$ in the Kleisli
  category for the symmetric-monoidal-groupoid monad $\SSS$: we have the
  equation $\DDD \X = \SSS \strictnerve \catO$, where $\DDD$ is upper
  decalage. The construction leads to a new characterisation of operadic
  categories, in which all the axioms end up as simplicial identities,
  and where the notion of operad over an operadic category takes the form
  of a simplicial map subject to well-known pullback conditions (the
  notion of IKEO map).
\end{abstract}

\tableofcontents

\section{Introduction}

\subsection*{Background}

Batanin and Markl~\cite{Batanin-Markl:1404.3886} introduced the notion of operadic
category as a general approach to operad-like structures and their algebras. An
operadic category is a small category $\catO$ with a chosen terminal object in
each connected component (chosen local terminals, for short), equipped with a
cardinality functor $\norm{ \thg } : \catO \to \Fin$ (the skeletal category
of finite sets), and with a fibre structure: for each morphism $f: T \to S$ in
$\catO$ and for each element $j \in \norm{S}$, there is assigned an object denoted
$f^{-1}(j)$ (not necessarily given by preimage). This structure has to satisfy a
number of natural axioms which mimic the properties of inverse images
of maps between finite sets. The category $\Fin$ in this definition is
the category of finite ordinals $\underline {n} = \{1<2< \ldots <n\}, n\ge 0 $ and
arbitrary maps between them. Each operadic category $\catO$ has its own notion of 
operad over it (called $\catO$-operads), and each $\catO$-operad has 
a category of algebras~\cite{Batanin-Markl:1404.3886}.
The operads over the terminal operadic category $\Fin$ are precisely ordinary
symmetric operads.

The theory of operadic categories is a rather powerful machinery, not only as a
unifying theory for operad-like structures, but also for more specific applications:
in the original paper~\cite{Batanin-Markl:1404.3886} it was used to prove the duoidal
Deligne conjecture, shortly after it was exploited to prove the Baez--Dolan stabilisation
hypothesis~\cite{Batanin:1511.09130}, and more recently
\cite{Batanin-Markl:2105.05198}, \cite{Batanin-Markl:1812.02935},
\cite{Batanin-Markl-Obradovic:2002.06640} it was shown how the theory of Koszul
duality can be developed abstractly in the setting of operadic categories (see
\cite{Batanin-Markl:2212.14598} for an application of this machinery to the
blob-complex in topological quantum field theories).

While the versatility of the theory and the general results achieved with it 
substantiate that operadic categories are a worthwhile notion, it is also fair to say
that the notion has been considered somewhat mysterious: the definition contains 
subtleties and unexpected strictnesses, a consequence of which is that the
notion of operadic category is not invariant under equivalence of categories! -- for
this reason it can be considered a combinatorial notion rather than a categorical one.
In parallel with the development of the theory and its applications, the definition
itself has therefore been under some scrutiny, aiming to fit it into more general
categorical viewpoints. Lack~\cite{Lack:1610.06282} gave an interpretation of operadic
categories in terms of the notion of skew-monoidal category of
Szlach\'anyi~\cite{Szlachanyi:1201.4981}, and Garner, Kock, and
Weber~\cite{Garner-Kock-Weber:1812.01750} recasted the definition in terms of algebras
for the (upper) decalage comonad $\DDD$ on $\kat{Cat}$ and related structures. The
present contribution inscribes itself into these theoretical developments, taking the
paper \cite{Garner-Kock-Weber:1812.01750} as a starting point.

One main point in the approach of
\cite{Garner-Kock-Weber:1812.01750} is the observation that the
chosen-local-terminals structure on a category $\mathcal{C}$ amounts
precisely to saying that $\mathcal{C}$ is a $\DDD$-coalgebra. In
particular, operadic categories are always $\DDD$-coalgebras. A key idea
put forth in the Garner--Kock--Weber paper is that the remaining
structure and axioms of operadic categories are also related to $\DDD$,
and that altogether operadic-category structure on a category should be a
kind of simplicial undecking. The comonad $\DDD$ induces a monad
$\widetilde\DDD$ on $\DDD\kat{-Coalg}$, whose algebras are {\em unary}
operadic categories, meaning operadic categories where all objects have
cardinality $1$, so that every morphism has only one fibre (see
Batanin--Markl~\cite{Batanin-Markl:2212.14598},
Hackney~\cite{Hackney:2312.00756} and Trnka~\cite{Trnka:2410.05064} for
further study of the unary case, which is interesting in itself). To
account for the multi-aspect -- morphisms with more than one fibre --
\cite{Garner-Kock-Weber:1812.01750} introduced a certain modification of
$\DDD$, defined on the arrow category of $\Cat$, and showed that operadic
categories (not only unary ones) can be described as algebras for the
monad induced on coalgebras, subject to one condition.

\subsection*{Contributions of this paper}

In the present paper we develop further the undecking idea from
\cite{Garner-Kock-Weber:1812.01750}. Again the starting point is the $\DDD$-coalgebra
viewpoint, but now the multi-aspect is encoded in a different way, namely in terms of
the symmetric-monoidal-groupoid monad $\SSS$, which acts levelwise on 
simplicial groupoids. In slogan form, our main theorem states
that

\begin{quote}
  {\em Operadic-category structure on a category amounts to an
  undecking of its nerve in the Kleisli category of $\SSS$.}
\end{quote}

The upshot of this is that all the structure involved in the notion of operadic
category $\catO$ --- chosen local terminals, cardinality functor, and fibres ---
integrates into a single simplicial object $\X$ called the {\em operadic nerve of
$\catO$} related to the original category $\catO$ by
$$
\DDD \X = \SSS \strictnerve \catO ,
$$
and remarkably, all the operadic-category axioms end up as simplicial identities.
Since $\SSS$ on a set is a groupoid, the simplicial object is a simplicial groupoid
(not just a simplicial set), but the only non-identity morphisms are symmetries coming
from $\SSS$. These symmetries are a subtle point of the theory, and in fact $\X$ is
slightly pseudo, as we shall detail.

{\em The operadic nerve.} This undecking amounts more precisely to the construction of
the {\em operadic nerve} $\nop(\catO)$, which is considerably more elaborate than the
ordinary nerve. It is a pseudo-simplicial groupoid instead of just a simplicial set,
and it is not Segal, only upper $2$-Segal (in the sense of Dyckerhoff and
Kapranov~\cite{Dyckerhoff-Kapranov:1212.3563}). The construction can be taken in two
steps. The first step is the observation from \cite{Garner-Kock-Weber:1812.01750} that
the chosen-local-terminals structure amounts to $\DDD$-coalgebra structure -- in fact
we work directly with chosen-local-terminals structure, interpreted simplicially: it
means to add extra top degeneracy operators to the ordinary nerve, and also add a new
set in degree $-1$, getting a top-split augmented simplicial set. The second step,
where we depart from \cite{Garner-Kock-Weber:1812.01750}, is to use the fibre
structure of the operadic category $\catO$ to define new top face operators. Since for
a map $T \to S$ in $\catO$ there is one fibre for each element $j \in \norm{S}$ in the
cardinality of the codomain, we do not get a single fibre but a $\norm{S}$-indexed
family of fibres, so the new top face operators land in $\SSS$ of the split augmented
simplicial set, hence necessitating the Kleisli viewpoint. Note in particular that
both the fibre structure and the cardinality functor of operadic categories are
involved in this second part of the undecking construction.

It turns out that the new top face operators do not strictly satisfy the simplicial
identities, but the operadic-category axioms can be used to establish 
instead the coherence
involved in the statement 

\bigskip

\noindent {\bf Coherence Theorem.} (\ref{NOP}) {\em The operadic nerve $\nop(\catO)$ is a 
pseudo-simplicial groupoid.}

\bigskip

This result is a
main technical difficulty of the undecking construction and involves some delicate
$2$-category theory. Jardine~\cite{JARDINE1991103} famously worked out 17 
families of simplicial $2$-identities required to guarantee coherence for pseudo-simplicial
objects. In the present case it is not that bad, because except for the top face
operators, the simplicial object is actually strict. As a consequence, out of the 17 
simplicial $2$-identities, 14
are actually strict identities. The crucial non-strict $2$-identity amounts to a cube
diagram of $2$-cells, and we can establish it first for the operadic category $\Fin$ of
finite sets, by exploiting its finer factorisation properties, and second in the
general case by exploiting the fibre axioms of operadic categories.
The proof for $\Fin$ is already quite involved. Rather than providing a long 
direct check, we have preferred to derive it from general theory. This theory
we have outsourced to a separate companion paper {\em Pita factorisation in operadic 
categories}~\cite{Batanin-Kock-Weber:2512.22794}, since it is a theory of independent 
interest.

The coherence theorem shows that although the fibres of an operadic
category are defined individually, in isolation, they nevertheless
assemble into a coherent structure, which is the operadic nerve.

\bigskip

{\em Operadic categories as simplicial `sets'.} So far we have explained
the construction of the operadic nerve. It is perhaps more surprising
that conversely every pseudo-simplicial groupoid $\X$ with the special
properties identified, uniquely defines an operadic category. The main
surprise here is that the cardinality functor can be reconstructed
although it is not directly visible in the simplicial structure. It turns
out it is actually hidden in the coherence $2$-cells for the simplicial structure.

In detail, the operadic-category axioms concerning the chosen local terminals are
derived from simplicial identities of the additional top degeneracy operators in $\X$.
The axioms concerning the fibre functor we derive from simplicial identities (and
$2$-identities) of the additional top face operators in $\X$. The structure
and axioms relating to the cardinality functor are not immediately visible, but they
are encoded via $\SSS$ in the pseudo-ness of the simplicial groupoid $\X$. In
particular we see that this pseudo-ness is not just something one can suppress or
strictify: it is forced upon us in order to accommodate the original notion of
operadic category.
Ultimately the pseudo-ness originates with the skeletal nature of $\Fin$.

\bigskip	
	
{\em Isomorphism of categories.}
The equality $\DDD \X = \SSS \strictnerve \catO$ is readily refined to the statement 
that we have a strictly commutative diagram of categories and functors
\begin{equation}\label{eq:not-pbk}
\begin{tikzcd}
\kat{OpCat}  \ar[d, "\operatorname{forget}"'] \ar[rr, "\nop"] && 
\kat{TopPs}
\ar[d, "\DDD"]  \\
\kat{Cat} \ar[r, "\strictnerve"'] & {[\simplexcategoryop,\Set]} \ar[r, "\SSS"'] & 
{[\simplexcategoryop,\SMGp]} .
\end{tikzcd}
\end{equation}
Here the middle category is just that of simplicial sets, and the next one that of
simplicial symmetric-monoidal groupoids. 
We stress that $\SSS$ acts on simplicial groupoids levelwise: 
$\SSS\strictnerve\CC$ is different from
the (nerve of the) free symmetric monoidal category on $\CC$.
The category $\kat{TopPs}$ has as
objects top-pseudo-simplicial symmetric monoidal groupoids, meaning pseudo-simplicial
symmetric monoidal groupoids whose only non-trivial $2$-cells are certain $\beta$-cells
involving consecutive top face operators.

Th square \eqref{eq:not-pbk} is not quite a pullback, for two reasons: one problem is
with degeneracy maps, and the second is a subtle issue with the $\beta$-cells. The
first problem is that in the actual pullback category, strictness of the top
degeneracy operators of $\X : \simplexcategoryop\to\SMGp$ cannot be guaranteed, as
would be required to match the local-terminals structure of an operadic category. The
decalage functor $\DDD$ thus throws away a bit too much. This problem can be fixed by
including those top degeneracy operators in the `category part'. This means that
instead of just plain categories we are dealing with categories with chosen local
terminals, which form the category $\Catlt$. These have a nerve functor $\ltnerve
:\Catlt \to [(\tEM)\op,\Set]$ landing in presheaves on $\tEM$, meaning they are
top-split augmented simplicial sets, or equivalently, simplicial sets with missing top
face operators. Let $\uuu: \tEM \to \simplexcategory$ be the inclusion. Instead of
applying to $\X$ the full decalage (throwing away degree zero as well as both top
faces and top degeneracies), we apply the gentler $\uuu\upperstar$ which only throws
away the top face operators. This adjustment provides a factorisation of the above
square as
$$
\begin{tikzcd}
  \kat{OpCat} \ar[d, "\operatorname{forget}"'] \ar[rr, "\nop"] 
  && \kat{TopPs}
  \ar[d, "\uuu\upperstar"]
  \\
  \kat{Cat}_{lt}  \ar[d, "\operatorname{forget}"'] \ar[r, "\ltnerve"'] & 
  {[(\tEM)\op,\Set]} \ar[r, "\SSS"']  \ar[d] & {[(\tEM)\op,\SMGp]} 
  \ar[d]
  \\
  \kat{Cat} \ar[r, "\strictnerve"'] & {[\simplexcategoryop,\Set]} \ar[r, "\SSS"'] & {[\simplexcategoryop,\SMGp]}  .
\end{tikzcd}
$$

The second problem has to do with a subtle condition enjoyed by
operadic nerves, namely that the $\beta$-cells are actually shuffle permutations.
We define
$\kat{TopPs}^{\operatorname{fr+sh}} \subset \kat{TopPs}$ to be the subcategory consisting
of top-pseudo-simplicial symmetric monoidal groupoids that are levelwise free,
and for which the $\beta$-cells are shuffle permutations. (Note that once $\X$ is 
required to be levelwise free, $\beta$ is essentially a permutation, so it makes 
sense to impose the shuffle condition.)

The pullback that actually works is 
$$
\begin{tikzcd}
  \kat{OpCat} \drpullback \ar[d, "\operatorname{forget}"'] \ar[rr, "\nop"] 
  && \kat{TopPs}^{\operatorname{fr+sh}}
  \ar[d, "\uuu\upperstar"]
  \\
  \kat{Cat}_{lt}   \ar[r, "\ltnerve"'] & 
  {[(\tEM)\op,\Set]} \ar[r, "\SSS"']  & {[(\tEM)\op,\SMGp]} 
\end{tikzcd}
$$
Note that the `levelwise free' condition is redundant for the sake of computing the 
pullback, since everything in the image of $\SSS$ is levelwise free, but 
it is necessary to impose it separately on $\kat{TopPs}$ in order to be able to 
formulate the shuffle condition.
The fact that the shuffle condition holds 
for any operadic nerve is a subtle feature of operadic categories coming from
the cardinality functor and the `fibrewise order-preserving' condition of pita 
factorisation in $\Fin$.

We have arrived at a very conceptual characterisation
of operadic categories, which is
our main theorem:

\medskip

\noindent {\bf Main Theorem.} (\ref{pbk-thm}) {\em The category of operadic categories is
isomorphic to the strict pullback just described. }

\medskip

Explicitly, the objects of this pullback are pairs $(\X,\CC)$
where $\X$ is a top-pseudo-simplicial symmetric strict monoidal groupoid satisfying 
the shuffle condition, 
$\CC$ is a small category with chosen
local terminals, and we have the strict equality $\uuu\upperstar \X = \SSS
\ltnerve \CC$.
The morphisms from $(\X,\CC)$ to $(\X',\CC')$ are given by pairs $(F,f)$ where $F: \X \to
\X'$ is a strict simplicial map and $f: \CC\to\CC'$ is a functor, required to be
compatible in the sense that this diagram commutes:
	\[
	\begin{tikzcd}[column sep ={8mm,between origins}]
	\uuu\upperstar \X \ar[d, "\uuu\upperstar(F)"'] & = & \SSS \ltnerve \CC \ar[d, 
	"\SSS\ltnerve(f)"] \\
	 \uuu\upperstar \X' & = & \SSS \ltnerve \CC'  .
	\end{tikzcd}
	\]

	\bigskip

{\em Bialgebra viewpoint.} The mix of strict and weak simplicial identities may seem a
bit strange, but there is a way to explain it conceptually. Namely, from
\cite{Garner-Kock-Weber:1812.01750} we know that the category $\Catlt$ is equivalent
to the category of (strict) $\DDD$-coalgebras
in $\Cat$. Presently we consider the decalage comonad on simplicial groupoids (or on 
category objects in $\Grpd$).
The decalage comonad $\DDD$ induces a
monad $\widetilde\DDD$ on $\DDD\kat{-Coalg}$ (in
simplicial groupoids or category objects in $\Grpd$). This monad commutes with the 
symmetric-monoidal-groupoid monad
$\SSS$, so as to lift to the Kleisli category for $\SSS$ on $\DDD\kat{-Coalg}$. The
category of operadic nerves can now be described as (equivalent to) the category of
shuffle-type
normal pseudo $\widetilde\DDD$-algebras in the Kleisli category for $\SSS$
(Theorem~\ref{thm:bialg}). That we are in the category of $\DDD$-coalgebras account
for the strictness of all the lower part of the simplicial structure, and the fact
that we consider {\em normal} pseudo algebras corresponds to the strictness of the
identities $d_n s_{n-1} = \id$. The remaining pseudo-ness and the coherence is
precisely what is means to be a pseudo-algebra. (We explain these viewpoints in
Section~\ref{sec:D-bialg}.)

\bigskip

{\em Operads over an operadic category.} The raison d'\^etre of operadic
categories are the operads over them~\cite{BMEH},
\cite{Batanin-Markl:1404.3886}, \cite{Batanin-Markl:2105.05198},
\cite{Batanin-Markl:1812.02935}. Originally
in~\cite{Batanin-Markl:1404.3886}, these were defined to be collections
of sets (or objects in a symmetric monoidal category) indexed by the
objects of $\catO$, with each morphism in $\catO$, defining
an operation from the list of sets associated to the fibres to the set
associated to the codomain. The translation of this notion to the setting
of operadic nerves turns out to be quite neat:\footnote{At least in cartesian
settings. To define operads in an arbitrary symmetric monoidal category, 
or in even more
  general settings, will require to develop a theory of internal operads inside
  categorical operads similar to what has been done in
  \cite{BMEH,Batanin-Berger:1305.0086} and \cite{Weber:1503.07585} for internal
  algebras of polynomial monads. We so not do
  that here, as it would take us too far afield and risk obscuring the simplicity of
  the idea.} 

\bigskip

\noindent {\bf Theorem.} (\ref{IKEOoperad}) {\em An operad (in $\Set$) over $\catO$ is
the same thing as another operadic nerve together with a strict simplicial map to 
$\nop(\catO)$,
which is \ikeo and levelwise free.}

\bigskip

\noindent
\ikeo~\cite{Galvez-Kock-Tonks:2409.03742} stands for `inner Kan' and
`equivalence on objects', a condition on simplicial maps well known in
simplicial homotopy theory. (While this theorem is elegant, the real
insight in the proof is the characterisation of operads in terms of
discrete operadic fibrations from
Batanin--Markl~\cite{Batanin-Markl:1404.3886}. We are only observing 
that that insight fits well into our simplicial viewpoint.)

\subsection*{Outlook}

In addition to the general pleasure of expressing things simplicially, there are
several further, more tangible benefits of this new simplicial approach to operadic 
categories.

One benefit is simply practical: in some situations it
is just easier to establish the conditions of the operadic nerve than to check the
old axioms by hand. In fact there is already an application of the Main Theorem,
namely in the work of Cebrian and Forero~\cite{Cebrian-Forero:2211.07721} who leverage
our theorem to construct operadic categories from directed hereditary species,
building on work of Schmitt~\cite{Schmitt:hacs},
G\'alvez--Kock--Tonks~\cite{Galvez-Kock-Tonks:1708.02570}, and
Carlier~\cite{Carlier:1903.07964}. Where Carlier (in the non-directed case) had to go 
through all the axioms by hand, Cebrian and Forero can bypass this by invoking 
our Main Theorem.

But we think the main benefits are of a more theoretical nature, namely that the
simplicial viewpoint opens up the door for weaker and higher dimensional notions
of operadic categories. Work in this direction is
underway (see for example the recent preprint of Markl and 
Trnka~\cite{Markl-Trnka:2601.20322} on weak unary operadic categories). We briefly 
indicate three directions.

{\em Weak operadic categories.} In view of the Main Theorem, we can now {\em define}
operadic categories as simplicial groupoids $\X$ with certain properties, the key
property being the equality of $\DDD \X$ with $\SSS$ of the nerve of a small
category. From here it is natural to weaken this definition by allowing an equivalence
instead of an equality. This gives a more flexible -- and equivalence-invariant! --
notion of operadic category, where in particular any category of finite sets (not just
skeletal) can serve as basis. Such an equivalence-invariant notion of operadic
categories was promised already in the original paper~\cite{Batanin-Markl:1404.3886},
but turned out to be tricky to flesh out. Only recently (partly inspired by the
present work) was one approach to the problem proposed by
Markl~\cite{Markl:2502.09163}. It should be noted that specifying both the 
category $\CC$ and the equivalence $\theta: \DDD \X \simeq \SSS 
\strictnerve \CC$ as part of the data is equivalent to discarding both 
$\CC$ and $\theta$ and demanding instead a property, namely the Segal 
condition.

The weak notion of operadic category (in terms of weak operadic nerves)
is general enough to serve as recipient for the two-sided bar
construction of an operad, for example as defined for polynomial
monads~\cite{Kock-Weber:1609.03276}. The image of this
two-sided-bar-construction functor is different from the image of the
operadic nerve functor from ordinary operadic categories: the two-sided
bar construction of an operad is a {\em strict} simplicial groupoid $\X$
(which is Segal, but sometimes not Rezk
complete~\cite{Kock-Weber:1609.03276}), whereas the operadic nerve is
generally pseudo and not generally Segal.

{\em Categorical operadic categories.} The simplicial approach allows an upgrade to a
notion of operadic $2$-category, essentially by replacing the usual nerve of a
$1$-category with the $2$-nerve of a $2$-category: a
categorical operadic category can be defined to be a simplicial category 
$\X$ with
$$
\DDD \X \simeq \SSS \mathfrak{N} \mathcal{K}  ,
$$
where $\mathcal{K}$ is a $2$-category and $\mathfrak{N}$ is the Duskin--Street nerve.
Some steps in this direction have been taken already by
Trnka~\cite{Trnka:2410.05064,Trnka:2506.12399}.

In forthcoming work
it is shown that
there is an operadic Grothendieck construction which to any operad over any operadic
category associates a categorical operadic category. Such categorical operadic
categories (or operadic $2$-categories) are important also for the theory of
ordinary operadic
$1$-categories, because they are the most general setting in which one can define operads
over an operadic category, rather than defining them in symmetric monoidal categories.
(There is a construction of an operadic $2$-category from any symmetric monoidal
category that makes the classical notion of operad in a symmetric monoidal category a
special case of this fancier notion.)

{\em Operadic $\infty$-categories.} Finally, the simplicial characterisation of
operadic categories can be copied almost verbatim to the realm of $\infty$-categories
in the sense of Segal spaces to give a definition of operadic $\infty$-category. Here
simplicial spaces stand in for both simplicial sets and simplicial groupoids, and
there is no need to say `pseudo': an operadic $\infty$-category is a simplicial space
$\X$ such that $\DDD \X$ is (levelwise) equivalent to $\SSS$ of a Segal space.
Examples of this notion come from symmetric $\infty$-operads (in the form of analytic
monads over $\infty$-groupoids~\cite{Gepner-Haugseng-Kock:1712.06469}, by way of an
$\infty$-version of the two-sided bar construction).

Beyond $\infty$-operads, an important motivation for this extension of the theory is
that the collection of configuration spaces (in the sense of Boavida and
Weiss~\cite{Boavida-Weiss:1502.01640}) is (strongly suspected to be) an operadic
$\infty$-category.

\section{Operadic categories}

\subsection*{The category $\Fin$}
We denote by $\Fin$ the skeletal category of finite sets, in which 
the objects are the finite ordinals $\un{n} = \{1,2,\ldots,n\}$ and 
the arrows are arbitrary set maps (not necessarily order-preserving).
Note that $\Fin$ has a \emph{unique} terminal
object $\un 1$, which we use to endow $\Fin$ with chosen local terminal
objects.

Given $g : \un n \rightarrow \un k$ in $\Fin$ and $i \in
\un k$, the inverse image $\{\,j \in \un n \mid g(j) = i\,\}$ is not in
general an object in the skeletal category $\Fin$, but as a subset of $\un n$
it inherits a linear order, and there is a unique order-preserving bijection with some object in
$\Fin$ which we denote $g^{-1}(i)$ and call the \emph{fibre of $g$ at $i$}.
We denote by
\begin{equation}\label{eq:10}
  \epsilon_{g, i} : g^{-1}(i) \rightarrow \un n
\end{equation}
the unique order-preserving injection whose image is $\{\,j \in \un n \mid
g(j) = i\,\}$. This is the pullback
\[
\begin{tikzcd}
g^{-1}(i) \drpullback \ar[d] \ar[r, "\epsilon_{g,i}"] & \un n \ar[d, "g"]  \\
\un 1 \ar[r, "\name{i}"'] & \un k
\end{tikzcd}
\]
in the category $\Fin$.

Given also
$f \colon \un m \rightarrow \un n$ in $\Fin$, we
write $f_{i}$ for the unique map of $\Fin$ rendering
\[
\begin{tikzcd}
(gf)^{-1}(i) \drpullback \ar[d, dotted, "f_i"'] \ar[r, "\epsilon_{gf,i}"] & 
\un m \ar[d, "f"]  \\
g^{-1}(i) \drpullback \ar[d] \ar[r, "\epsilon_{g,i}"] & \un n \ar[d, "g"]  \\
\un 1 \ar[r, "\name{i}"'] & \un k
\end{tikzcd}
\]
commutative, and call it the \emph{fibre map of $f$ with respect to
  $g$ at $i$}.

\bigskip

The order-preserving injections
$\epsilon_{g, i} : g^{-1}(i) \rightarrow \un n$
assemble into a single bijection
\begin{equation}\label{pi}
\pi_g : \un n \isopil \sum_i g^{-1}(i)
\end{equation}
expressing $\un n$ as a sum of its fibres. This provides altogether the {\em
pita factorisation}~\cite{Batanin-Kock-Weber:2512.22794}:  {\em every arrow 
$g: \un n \to
\un k$ factors uniquely as
\[
\begin{tikzcd}
\un n \ar[dd, "g"'] \ar[rd, "\pi_g"] &   \\
 & \un n' \ar[ld, "\eta_g"] &[-33pt] = \sum_i g^{-1}(i)  \\
 \un k &
\end{tikzcd}
\]
where $\eta_g$ is order-preserving and $\pi_g$ is bijective and 
fibrewise order-preserving.}
Note that here $\un n'$ is the same object as $\un n$ (since $\Fin$ is 
skeletal), but it is practical to keep it distinct in the notation to remember that the
natural identification with $\un n$ is given by $\pi_g$, not by the identity map.

The pita factorisation does not constitute a factorisation system, because the
left-hand component is not described intrinsically (the condition `order-preserving on
fibres' is local to the codomain), but it nevertheless has many nice properties, and
is functorial to some extent~\cite{Batanin-Kock-Weber:2512.22794}: for any composable pair
of maps $\un m \stackrel{f}\to \un n \stackrel{g}\to \un k$,
if both $gf$ and $g$ are
pita-factorised, then there is a unique comparison map $\eta(f/g)$ as in this diagram:
\begin{equation}\label{func-pita}
\begin{tikzcd}[column sep={24pt,between origins}, row sep={36pt,between origins}]
\un m \ar[rrrr, "f"] \ar[rd, "\pi(gf)"']&&&& \un n \ar[ld, "\pi(g)"] \\
& \cdot \ar[rr, dashed, "\eta(f/g)"] \ar[rd, "\eta(gf)"']&& \cdot \ar[ld, "\eta(g)"] & \\
&& \un k .&&
\end{tikzcd}
\end{equation}
Also, the factorisation interacts well with ordinal sum: for 
morphisms $g_1$ and $g_2$ we have
\begin{eqnarray*}
\pi(g_1+g_2) &=& \pi(g_1)+\pi(g_2)\\
\eta(g_1+g_2) &=& \eta(g_1)+\eta(g_2) \,,
\end{eqnarray*}
and the same goes for sums of maps $\eta(f/g)$ arising for composable 
chains as above.

Pita factorisation is an important technical aspect of the theory
of operadic categories. In the companion paper 
\cite{Batanin-Kock-Weber:2512.22794} we study pita factorisation in more 
general operadic categories, and prove a general coherence result 
(reproduced below as Proposition~\ref{pitapaper-prop}), which 
we rely upon in the present paper.  The above basic results can also be 
found in \cite{Batanin-Kock-Weber:2512.22794}.

\subsection*{Axioms for operadic category}

Since we are going to dig into the details of the axioms, we recall the
definition of operadic category of Batanin and
Markl~\cite{Batanin-Markl:1404.3886}. The notion can be seen as
specifying a category with \emph{formal} notions of fibre and fibre map.
The fibres of a map need not be subobjects of the domain as in the case
of $\Fin$, but the axioms will ensure that they retain many important
properties of fibres in $\Fin$.

\begin{definition}
  \label{def:operadic-category}\cite{Batanin-Markl:1404.3886}
  An \emph{operadic category} is given by the following data:
  \begin{enumerate}[label=(D\arabic*)]
  \item \label{data:operadic-Q1} A category $\catO$ endowed with chosen local
    terminal objects; we write $\tau_R:R \to U_R$ for the unique 
    morphism to the chosen terminal object of the component that $R$ belongs 
    to.
    \item \label{data:operadic-Q2} A \emph{cardinality functor}
      $\norm{\thg} \colon \catO \to \Fin$;
    \item \label{data:operadic-Q3} For each object $R\in \catO$ and each
      $i \in \norm{R}$ a \emph{fibre functor}
      $$\fib_{R,i} \colon \catO/R \to \catO$$
      whose action on objects and morphisms we denote as follows:
      \begin{align*}
        \cd[@C1.3em]{
          S \ar[rr]^-{g} && R
        } \qquad &\mapsto \qquad g^{-1}(i)\\
        \cd[@C1em@R-0.7em]{{T} \ar[rr]^-{f} \ar[dr]_-{gf} & &
          {S} \ar[dl]^-{g} \\ &
          {R}} \qquad &\mapsto \qquad
        f^g_{i} \colon (gf)^{-1}(i) \to g^{-1}(i)\rlap{ ,}
      \end{align*}
      referring to the object $g^{-1}(i)$ as the \emph{fibre
      of $g$ at $i$}, and the morphism
    $f_{i} \colon (gf)^{-1}(i) \to g^{-1}(i)$ as the \emph{fibre
      map of $f$ with respect to $g$ at~$i$};
  \end{enumerate}
  all subject to the following axioms, where
  in~\ref{axQ:BM-fibres-of-local-fibres}, we write $\epsilon j$
  for the image of $j \in {\norm g}^{-1}(i)$ under the map
  $\epsilon_{\norm g, i} \colon {\norm g}^{-1}(i) \rightarrow \norm S$
  of~\eqref{eq:10}:

  \begin{enumerate}[label=(A\arabic*)]

	\item \label{axQ:BM-abs(lt)} If $R$ is a local terminal then
    $\norm{R}=\un 1$;

	\item \label{axQ:BM-fibres-of-identities} For all $R \in \catO$ and
    $i \in \norm R$, the object $(\id_R)^{-1}(i)$ is chosen local terminal;

	\item \label{axQ:BM-67} For all $g \in \catO / R$ and $i \in \norm R$,
    one has $\norm{\smash{g^{-1}(i)}} = {\norm g}^{-1}(i)$, while for
    all $f \colon gf \rightarrow g$ in $\catO / R$ and $i \in \norm R$,
    one has $\norm{\smash{f_{i}}} = \smash{\norm{f}_{i}}$;

	\item \label{axQ:BM-fibres-of-tau-maps} For $R \in \catO$ and for the unique 
	element $1\in \norm{U_R}$, one has
    $\tau_R^{-1}(1) = R$, and for $g \colon S \to R$, one has
    $g_1 = g$;

	\item \label{axQ:BM-fibres-of-local-fibres} For
    $f \colon gf \rightarrow g$ in $\catO/R$, $i \in \norm R$ and
    $j \in \norm{g}^{-1}(i)$, one has that
    $(f_i)^{-1}(j) = f^{-1}(\epsilon j)$, and given also
    $h \colon gfh \rightarrow gf$ in $\catO / R$, one has
    $(h_i)_{j} = h_{\epsilon j}$.
  \end{enumerate}
  
\end{definition}
  
The preceding definitions are exactly those
of~\cite{Batanin-Markl:1404.3886} with only some slight repackaging changes
as in \cite{Garner-Kock-Weber:1812.01750}. For subtle variations on the axioms, and 
slightly different notions of operadic category,
see Lack~\cite{Lack:1610.06282} and Markl~\cite{Markl:2502.09163}.

The category $\Fin$ has a unique (and hence chosen) terminal object.
Axiom~\ref{axQ:BM-abs(lt)} now says that the cardinality functor 
preserves chosen local terminals.
Any slice category has a canonical choice of terminal object, 
namely the identity arrow. With respect to this choice, Axiom~\ref{axQ:BM-fibres-of-identities}
says that the fibre functor preserves chosen local terminals.
Axiom~\ref{axQ:BM-67} says that 
for $R \in \catO$ and $i \in \norm R$, the square 
\begin{equation}
\label{diag:fibres-and-cardinalities}
\begin{aligned}{ 
\xygraph{!{0;(2,0):(0,.5)::}
{\catO_{\smash{/R}}}="p0" [r] {\catO}="p1" [d] {\Fin}="p2" [l] {\Fin_{\smash{/|R|}}}="p3"
"p0" :"p1"^-{\fib_{R,i}} :"p2"^-{\norm{\thg}} 
:@{<-}"p3"^-{\fib_{|R|,i}} :@{<-}"p0"^-{\norm{\thg}_{/R}}}}
\end{aligned}
\end{equation}
commutes.
Axiom~\ref{axQ:BM-fibres-of-tau-maps} says that $\fib_{R,i}$ is the
domain functor whenever $R \in \catO$ is chosen local terminal.
Intuitively, the first clause of \ref{axQ:BM-fibres-of-local-fibres}
identifies the fibres of the fibre maps of a map, with the fibres of
that map. 
The second part of \ref{axQ:BM-fibres-of-local-fibres} says that the fibre maps of
the fibre maps of a map are themselves fibre maps of that map. For the 
details of these interpretations, 
see~\cite{Garner-Kock-Weber:1812.01750}.

  \begin{definition} A functor $F \colon \catP \rightarrow \catO$ between operadic
  categories is called a \emph{strict operadic functor} when it
  \begin{enumerate}
	\item strictly commutes with the cardinality functors to $\Fin$, 
	\item strictly preserves local terminal objects,  
	\item strictly preserves fibres and fibre maps, in the sense that
	\begin{equation*}
	  F(g^{-1}(i)) = (Fg)^{-1}(i) \qquad \text{and} \qquad F(f_{i}) =
	  (Ff)_{i}
	\end{equation*}
	for all $f \colon gf \rightarrow g$ in $\catP/R$ and $i \in \norm R$. 
  \end{enumerate}
  
  We write $\OpCat$ for the category of operadic categories and strict
  operadic functors.
\end{definition}

The definition of operad over an operadic category 
(cf.~\cite{Batanin-Markl:1404.3886}) will be postponed to 
Section~\ref{sec:operads}.

\section{Split augmented simplicial objects}

We need to set up some notation, and recall a few basic facts.

As always, let $\simplexcategory$ be the category of non-empty finite standard
ordinals and order-preserving maps. We shall need also two other base categories
closely related to $\simplexcategory$. Let $\tEM$ be the category of non-empty finite
standard ordinals with a top element and top-preserving order-preserving maps. Let
$\tKLEISLI$ be the full subcategory of $\tEM$ consisting of those objects where the
top element is not the only element.

The forgetful functor $\uuu :\tEM \to \simplexcategory$ has a left adjoint which
freely adds a top element. Let $\ttt$ denote the monad on $\simplexcategory$ induced
by this adjunction. The forgetful functor $\uuu$ is in fact monadic, so that $\tEM$
becomes the category of Eilenberg--Moore algebras for $\ttt$, and $\tKLEISLI$ is the
Kleisli category (the full subcategory of free algebras), hence the notation for these
two categories.

The left adjoint free functor thus factors as
\[
\begin{tikzcd}[column sep={3em,between origins}]
& \tEM  \ar[dd, "\uuu"] \\
\tKLEISLI \ar[ru, "\iii"]  \ar[r, phantom, "\isleftadjointto" description] & {}\\
& \simplexcategory \ar[lu, "\kkk"]
\end{tikzcd}
\]
where $\kkk$ is identity-on-objects and $\iii $ is fully faithful.
We let the functors $\kkk$ and $\iii$ dictate the naming conventions for the objects in the 
three categories: we thus have
$$
\kkk[n] = [n] \qquad \iii[n] = [n] \qquad \uuu[n]=[n+1] .
$$
The category $\tEM$ has an extra object denoted $[-1]$ corresponding to the
linear-order-with-top-element consisting of only the top element, but because of the 
$+1$ in the definition, the functor $\uuu$ is still bijective on objects.

\medskip

We now consider presheaves with values in some category $\EE$, denoting 
presheaf categories by $\PrSh$.

Objects in $\PrSh(\simplexcategory)$ are of course simplicial objects;
we picture them by drawing the first few face and degeneracy operators
\begin{center}
  \begin{tikzcd}[column sep={7em,between origins}]
    \phantom{\X_{-1}} & {\X_0} & {\X_1} & {\X_2} & {}
    \arrow["{s_0}"{description, inner sep = .5pt}, shorten <=5pt, shorten >=5pt, from=1-2, to=1-3]
    \arrow["{d_1}"', shift right=2.5, from=1-3, to=1-2]
    \arrow["{d_0}", shift left=2.5, from=1-3, to=1-2]
    \arrow["{s_1}"{description, inner sep = .5pt, pos=0.4}, shift left=2.5, shorten <=5pt, 
	shorten >=5pt, from=1-3, to=1-4]
    \arrow["{s_0}"{description, inner sep = .5pt, pos=0.4}, shift right=2.5, shorten <=5pt, 
	shorten >=5pt, from=1-3, to=1-4]
    \arrow["{d_1}"{description, inner sep = .5pt, pos=0.4}, from=1-4, to=1-3]
    \arrow["{d_2}"', shift right=5, from=1-4, to=1-3]
    \arrow["{d_0}", shift left=5, from=1-4, to=1-3]
    \arrow["\cdots"{description}, phantom, from=1-5, to=1-4]
  \end{tikzcd}
\end{center}

Objects in $\PrSh(\tKLEISLI)$ are called {\em top-split simplicial
objects}; the diagram of their generating face and degeneracy operators starts like this:
\begin{center}
  \begin{tikzcd}[column sep={7em,between origins}]
    \phantom{\X_{-1}} & {\X_0} & {\X_1} & {\X_2} & {}
    \arrow["{s_0}"{description, inner sep = .5pt}, shorten <=5pt, shorten >=5pt, from=1-2, to=1-3]
    \arrow["{d_1}"', shift right=2.5, from=1-3, to=1-2]
    \arrow["{d_0}", shift left=2.5, from=1-3, to=1-2]
    \arrow["{s_1}", bend left=25, shift left=2, from=1-2, to=1-3, shorten <=5pt, shorten >=5pt]
    \arrow["{s_1}"{description, inner sep = .5pt, pos=0.4}, shift left=2.5, shorten <=5pt, 
	shorten >=5pt, from=1-3, to=1-4]
    \arrow["{s_0}"{description, inner sep = .5pt, pos=0.4}, shift right=2.5, shorten <=5pt, 
	shorten >=5pt, from=1-3, to=1-4]
    \arrow["{d_1}"{description, inner sep = .5pt, pos=0.4}, from=1-4, to=1-3]
    \arrow["{d_2}"', shift right=5, from=1-4, to=1-3]
    \arrow["{d_0}", shift left=5, from=1-4, to=1-3]
    \arrow["\cdots"{description}, phantom, from=1-5, to=1-4]
    \arrow["{s_2}", bend left=25, shift left=5, from=1-3, to=1-4, shorten <=5pt, shorten >=5pt]
  \end{tikzcd}
\end{center}
The extra top sections satisfy the simplicial identities corresponding to letting
the indices go one higher than usual for degeneracy operators.

Objects in $\PrSh(\tEM)$ are called {\em top-split augmented
simplicial objects}; they are drawn with
\begin{center}
  \begin{tikzcd}[column sep={7em,between origins}]
    {\X_{-1}} & {\X_0} & {\X_1} & {\X_2} & {}
    \arrow["{d_0}", from=1-2, to=1-1]
    \arrow["{s_0}", bend left=25, shift right=1, shorten <=5pt, shorten >=5pt, from=1-1, to=1-2]
    \arrow["{s_0}"{description, inner sep = .5pt}, shorten <=5pt, shorten >=5pt, from=1-2, to=1-3]
    \arrow["{d_1}"', shift right=2.5, from=1-3, to=1-2]
    \arrow["{d_0}", shift left=2.5, from=1-3, to=1-2]
    \arrow["{s_1}", bend left=25, shift left=2, from=1-2, to=1-3, shorten <=5pt, shorten >=5pt]
    \arrow["{s_1}"{description, inner sep = .5pt, pos=0.4}, shift left=2.5, shorten <=5pt, shorten >=5pt, from=1-3, to=1-4]
    \arrow["{s_0}"{description, inner sep = .5pt, pos=0.4}, shift right=2.5, shorten <=5pt, shorten >=5pt, from=1-3, to=1-4]
    \arrow["{d_1}"{description, inner sep = .5pt, pos=0.4}, from=1-4, to=1-3]
    \arrow["{d_2}"', shift right=5, from=1-4, to=1-3]
    \arrow["{d_0}", shift left=5, from=1-4, to=1-3]
    \arrow["\cdots"{description}, phantom, from=1-4, to=1-5]
    \arrow["{s_2}", bend left=25, shift left=5, from=1-3, to=1-4, shorten <=5pt, shorten >=5pt]
  \end{tikzcd}
\end{center}

\begin{proposition}\label{splitsimp}
  When the category $\EE$ is idempotent-complete, the pair of adjoint functors
  $(\iii\op)_!\dashv (\iii\op)\upperstar$ defines an adjoint equivalence between
  split simplicial objects and split augmented simplicial objects in $\EE$.
\end{proposition}
	
\begin{proof}
  Let $\X: (\tKLEISLI)\op \to \EE $ be a split simplicial object. We are going to show that
  $(\iii\op)_!$ can be constructed as an absolute coequaliser in such a way that $\X$
  extends further to $(\tEM)\op$ in an essentially unique way, using the assumption that
  idempotents split
  in $\EE$. Given a diagram in $\EE$ as on the left
  \begin{equation*}
    \xygraph{!{0;(2,0):(0,.5)::}
      {A}="p1" [r] {B}="p2"
      "p1" :@<2.25ex>"p2"|-{s_1} :@<-.75ex>"p1"|-{d_1} "p2":@<.75ex>"p1"|-{d_0}
      "p2" [r] 
      {Q}="q0" [r] {A}="q1" [r] {B}="q2"
      "q1" :@<2.25ex>"q2"|-{s_1} :@<-.75ex>"q1"|-{d_1} "q2":@<.75ex>"q1"|-{d_0}
      "q1" :"q0"|-{q} "q0" :@<1.5ex>"q1"|-{i}
    }
  \end{equation*}
  such that $d_1s_1 = \id_A$ and $d_0s_1d_1 = d_0s_1d_0$, then $d_0s_1$ is idempotent.
  Splitting this idempotent produces $Q$, $q$ and $i$ as on the right in the previous
  display, such that $qi = \id_Q$ and $iq = d_0s_1$. Note moreover that $q$ is the
  coequaliser of $(d_1,d_0)$ since $(i,s_1)$ exhibits it as a split coequaliser (see
  \cite{MacLane:categories} Chapter VI.6).
	
  Applied to the present situation, this says that the only way
  to extend $\X$ further to $(\tEM)\op \to \EE$ is to split the idempotent
  $d_0s_1$ to produce the required extra data. We denote the resulting extension, at
  the level of generating morphisms, as
  \begin{equation}\label{diag:all-simp-data-D-coalgebra}
    \begin{aligned}
      \xygraph{!{0;(2,0):(0,.5)::}
          {\pi_0(\X)}="p0" [r] {\X_0}="p1" [r] {\X_1}="p2" [r] {\X_2}="p3" [r(.5)] {...}="p4"
          "p1" :@<3ex>"p2"|-{s_1} :@<-1.5ex>"p1"|-{d_1} :"p2"|-{s_0} :@<1.5ex>"p1"|-{d_0}
          "p2" :@<4.5ex>"p3"|-{s_2} :@<-3ex>"p2"|-{d_2} :@<1.5ex>"p3"|-{s_1} :"p2"|-{d_1} :@<-1.5ex>"p3"|-{s_0} :@<3ex>"p2"|-{d_0}
            "p1" :"p0"|-{q} :@<1.5ex>"p1"|-{i}}
    \end{aligned} \qedhere
  \end{equation}
\end{proof}

When $\EE=\Set$ and $\X$ is the nerve of a category, note that $\pi_0(\X)$, as the
coequaliser of $(d_1,d_0)$, is the set of connected components of the category $\X$.

\bigskip

{\em Decalage.} The adjunction 
\[
\begin{tikzcd}[column sep={3em,between origins}]
& \PrSh(\tEM) \ar[ld, "\iii \upperstar"'] \\
\PrSh(\tKLEISLI) \ar[rd, "\kkk \upperstar"'] \ar[r, phantom, "\isleftadjointto" description] & {}\\
& \PrSh(\simplexcategory) \ar[uu, "\uuu \upperstar"']
\end{tikzcd}
\]
defines a comonad on $\PrSh(\simplexcategory)$ which is the {\em (upper) decalage
comonad}
$$
\DDD := \ttt\upperstar = (\iii \kkk )\upperstar \uuu \upperstar : \PrSh(\simplexcategory) \to \PrSh(\simplexcategory).
$$

There is an isomorphism of categories
$$
\DDD\kat{-Coalg} \simeq \PrSh(\tKLEISLI) .
$$
This follows from general principles, namely the universal property of the Kleisli 
object~\cite{Street:formal-monads}. (It is also instructive to work out by hand:
In
explicit terms, a $ \tKLEISLI$-presheaf $\A$ has an underlying simplicial space
$\X:=\kkk \upperstar (\A)$ possessing extra top degeneracies; these
assemble into a simplicial map $\gamma: \X \to \DDD(\X) $ which is the
structure map of a $\DDD$-coalgebra: the split-simplicial identities
satisfied by the extra top degeneracies correspond precisely to the coalgebra
axioms.)

Note that decalage restricts from simplicial sets to $\Cat$. The next result describes
the $\DDD$-coalgebras in $\Cat$.

	\begin{proposition}[Garner--Kock--Weber~\cite{Garner-Kock-Weber:1812.01750}]
		\label{prop:Dcoalg}
	
		For a category $\catO$ the following structures are equivalent:
		\begin{enumerate}
			\item $\DDD$-coalgebra structure on $\catO$ (or on $\strictnerve\catO$).
			\item A choice of local terminal objects in $\catO;$
			\item Splitting structure on the nerve $\strictnerve\catO;$
			\item An augmented splitting structure on $\strictnerve\catO$;
		\end{enumerate}
	\end{proposition}
	
	The extension is given as follows: let $\catO_{-1}:=\pi_0(\catO)$ be the
	set of connected components. 
	Then the chosen-local-terminals structure defines a split augmented simplicial set
	$\ltnerve(\catO)$
  \begin{equation}\label{augmentedsplit}
    \xymatrix@C=6em{\catO_{-1} \ar@/^1em/[r]|-{s_0}     
      & \catO_0 \ar[l]|-{d_0}     
      \ar[r]|-{s_0}
      \ar@/^2em/[r]|-{s_1}
      & \catO_1   \ar@/^1em/[l]|-{d_0}  \ar@/_1em/[l]|-{d_1}  
      \ar@/_1em/[r]|-{s_0} \ar@/^1em/[r]|-{s_1} \ar@/^3em/[r]|-{s_2}
      &\catO_2   \ar[l]|-{d_1}\ar@/^2em/[l]|-{d_0}\ar@/_2em/[l]|-{d_2}
      &\hskip -10em\cdots 
    }
  \end{equation}
  The degeneracy $s_0:\catO_{-1}\to \catO_0$ returns the chosen local
  terminal of a given component, and the other top degeneracy operators
  $s_r:\catO_{r-1}\to \catO_r$ append to a chain the unique
  morphism from it to the
  chosen terminal object of its component.

  Let $\Catlt$ denote the category small categories with chosen local terminals
  (and functors that preserve chosen local terminals). We have the following
  functors
\[
\begin{tikzcd}
  \OpCat \ar[d, "\operatorname{forget}"'] & \\
\Catlt 
\ar[d, "\operatorname{forget}"'] \ar[r, "\nerexs"] & \PrSh(\tEM) \ar[d, "(\iii\kkk)\upperstar"]  \\
\Cat \ar[r, "\strictnerve"'] & \PrSh(\simplexcategory) .
\end{tikzcd}
\]

\section{The operadic nerve of an operadic category}
\label{sectionoperadicnerve}

The aim of this section is to construct a functor, called the {\em operadic nerve
functor}
$$
\nop:\OpCat \to { [\simplexcategoryop,\SMGp]_{\operatorname{ps}}} .
$$
Here ${ [\simplexcategoryop,\SMGp]_{\operatorname{ps}}}$ is the category of 
pseudo-simplicial symmetric strict
monoidal groupoids and their strict morphisms (so all morphisms are symmetric
strict monoidal, strictly preserving coherence $2$-cells).

Recall that the symmetric-monoidal-groupoid monad $\SSS:
\Grpd\to\Grpd$ is given by sending a groupoid $\CC$ to the groupoid $\SSS\CC$
whose objects are the
words in $\CC$, and whose morphisms from $(x_i)_{i \in \un n}$ to
$(y_i)_{i \in \un n}$ are given by {\em decorated permutations},
meaning pairs $(\rho,(f_i)_{i \in \un n})$ where $\rho \in
\mathfrak{S}_n$ is a permutation of $\un{n}$ and $f_i : x_i \to
y_{\rho i}$ is an arrow in $\CC$ for each $i \in \un{n}$. More
formally, the data $(\rho,f)$ is precisely a $2$-cell
\[
\begin{tikzcd}[column sep={3.2em,between origins},row sep={3.6em,between origins}]
	\un n \ar[rd, "x"'] \ar[rr, "\sim"', "\rho"] & \ar[d, phantom, 
	"\begin{rotate}{24}{$\Rightarrow$}\end{rotate}\;\;_{f}" description] & \un n \ar[ld, "y"]  \\
	& \CC &
\end{tikzcd}
\]

We extend $\SSS$ to a monad 
on simplicial groupoids as follows. If
$\X:\simplexcategoryop\to\Grpd$ is a simplicial groupoid, then $\SSS \X$ is
just the composite functor $\simplexcategoryop\xto{\X}\Grpd\xto{\SSS} \Grpd$. Since
$\SSS$ preserves both strict pullbacks and homotopy pullbacks, it 
preserves the conditions of being a category object or being a decomposition 
space.
Henceforth we use the symbol $\SSS$ in this levelwise
   sense. We just warn that for $\strictnerve\CC$ the
   nerve of a category, $\SSS\strictnerve\CC$ in this
   levelwise sense is different from the (nerve of the)
   free symmetric monoidal category on $\CC$.

\bigskip

Let now $\catO$ be an operadic category. Since we have chosen local
terminals in $\catO$, we can apply the local-terminals nerve 
$$
\nerexs : \OpCat \to [(\tEM)\op,\Set]
$$
to get the diagram \eqref{augmentedsplit}.
Applying $\SSS$ levelwise to
\eqref{augmentedsplit} we obtain a split augmented simplicial symmetric strict
monoidal groupoid
\[
\xymatrix@C=5.8em{\SSS\catO_{-1}\ar@/^1em/[r]|-{\SSS s_0}     
	& \SSS\catO_0 \ar[l]|-{\SSS d_0}     
	\ar[r]|-{\SSS s_0}
	\ar@/^2em/[r]|-{\SSS s_1}
	& \SSS\catO_1   \ar@/^1em/[l]|-{\SSS d_0}  \ar@/_1em/[l]|-{\SSS d_1}  
	\ar@/_1em/[r]|-{\SSS s_0} \ar@/^1em/[r]|-{\SSS s_1} \ar@/^3em/[r]|-{\SSS s_2}
	&\SSS\catO_2   \ar[l]|-{\SSS d_1}\ar@/^2em/[l]|-{\SSS d_0}\ar@/_2em/[l]|-{\SSS d_2}
	&\hskip -10em\cdots
}
\]

\begin{definition} For $r>0$ the groupoid of $r$-simplices $\nop(\catO)_r$ is the groupoid
  $\SSS\catO_{r-1}$ i.e.~the free symmetric monoidal groupoid generated by the
  set of $(r{-}1)$-chains of composable morphisms in $\catO$. In 
  degree zero, we
  have canonical identifications of groupoids $\nop(\catO)_0 = \SSS \catO_{-1}= \SSS
  \pi_0(\catO) = \SSS \catO_{\mathtt{u}}$ (here $\catO_{\mathtt{u}}$ denotes the
  set of chosen local terminal objects of $\catO$).
  
  { We already have an augmented split \emph{strict} simplicial structure
      on $\nop(\catO)$ induced by $\SSS\nerexs(\catO)$.}
  We thus define face and degeneracy operators in $\nop(\catO)$ as
  $$
  \dd_i := \SSS d_i:{\nop(\catO)}_r\to \nop(\catO)_{r-1}, \  \text{for} \ 0\le i\le r-1 ,
  $$ 
  $$
  \ss_i := \SSS s_i:\nop(\catO)_{r-1}\to \nop(\catO)_{r}, \  \text{for} \ 0\le i\le r-1 .
  $$  
\end{definition} 

We will use the fibre functor of the operadic category $\catO$ to define the top
face operators
$$
\dd_{r}:\nop(\catO)_{r}\to \nop(\catO)_{r-1}\,,
$$
so the operadic nerve of $\catO$ will be a symmetric monoidal
pseudo-simplicial groupoid of the form:

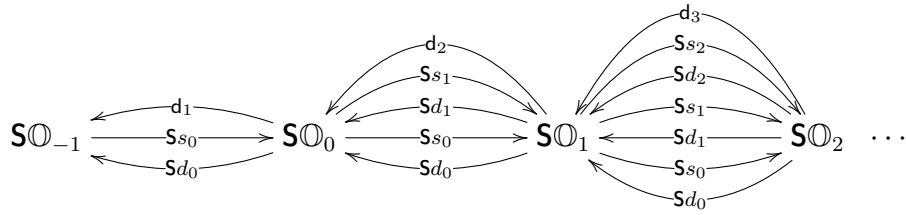
\[
\xymatrix@C=5.8em{\SSS\catO_{-1} \ar[r]|-{\SSS s_0}     
	& \SSS\catO_0
	\ar[r]|-{\SSS s_0} \ar@/^1em/[l]|-{\SSS d_0}\ar@/_1em/[l]|-{\dd_1}
	\ar@/^2em/[r]|-{\SSS s_1}
	& \SSS\catO_1   \ar@/^1em/[l]|-{\SSS d_0}  \ar@/_1em/[l]|-{\SSS d_1}
	\ar@/_3em/[l]|-{\dd_2}  
	\ar@/_1em/[r]|-{\SSS s_0} \ar@/^1em/[r]|-{\SSS s_1} \ar@/^3em/[r]|-{\SSS s_2}
	&\SSS\catO_2   \ar[l]|-{\SSS d_1}\ar@/^2em/[l]|-{\SSS d_0}\ar@/_2em/[l]|-{\SSS d_2}
	\ar@/_4em/[l]|-{\dd_3}  
	&\hskip -10em\cdots
}
\]

By Jardine's {\em supercoherence theorem}~\cite{JARDINE1991103}, a pseudo-simplicial 
groupoid can be specified by generators and relations: the generators are the usual 
face and degeneracy operators in categorical dimension $1$, together with invertible 
$2$-cells instead of the usual simplicial identities. The relations are given by 17 
families of equations satisfied by these $2$-cells \cite[(1.4.1)--(1.4.17)]{JARDINE1991103}.
In our case most of the simplicial identities hold strictly. In fact there is only one
family of non-trivial $2$-cells $\beta$, and there are only three families of equations 
to check. These can be interpreted in terms of pseudo-algebras for a certain monad 
$\widetilde\DDD$ as we explain in Section~\ref{sec:D-bialg}.

We begin from $\dd_1:\nop(\catO)_{1}\to \nop(\catO)_{0}$. By definition,
$\nop(\catO)_1 = \SSS\catO_0$. Since $\dd_1$ is required to be a symmetric
strict monoidal functor, it is enough to define it on a monoidal 
generator $T\in \SSS\catO_0$, that is on an object of $\catO$. We define
it to be the list of connected components of fibres of the
identity $\id:T\to T$.

For $r>1$, let $p= (T_r\stackrel{f_{r-1}}{\to}
T_{r-1}\stackrel{f_{r-2}}{\to} \ldots \stackrel{f_1}{\to} T_1)$ be a
generator of $\nop(\catO)_r$, that is, an object of $\catO_{r-1}$. We
then have a commutative diagram in $\catO$:
\begin{equation*} 
	p: \qquad
	\xymatrix@C = +3em@R = +2em{
		T_r     \ar[r]^{f_{r-1}} \ar@/^-3.3ex/[drr]_{h_{r-1}} &  T_{r-1}
		\ar[r]^{f_{r-2}}  \ar@/^-1.5ex/[dr]_{h_{r-2}} & \ldots  \ar[r]^{f_2} \ar[d] &   T_2
		\ar@/^+1.3ex/[dl]^{h_1:=f_1}
		\\
		&&T_1 &
	}
\end{equation*} 
By functoriality of the fibre functor \ref{data:operadic-Q3}, this
diagram induces for each $i\in |T_1| =: k$ a chain of morphisms between
fibres:
$$
p_i = (h_{r-1}^{-1}(i)\xto{(f_{r-1})_i} \ldots \xto{(f_{2})_i} h_{1}^{-1}(i)).
$$
We define $\dd_r(p)$ to be this list of $(r{-}2)$-chains, which is thus 
an object of $\nop(\catO)_{r-1} = \SSS\catO_{r-2}$. In short:
$$
\dd_r(p) = (p_1,p_2,\ldots,p_k) .
$$

\begin{blanko}{$2$-cells of the operadic nerve.}\label{beta=pi}
  To complete the definition of $\nop(\catO)$ as a {\em pre-simplicial}
  groupoid (that is, containing all the data of $d$, $s$ and $\beta$, but not yet the 
  equations on $\beta$), we need to define the $2$-cells $\beta$.

  We put all $2$-cells not involving a top face operator equal
  to the identities. The general $2$-cells for the top face operators and
  degeneracies are the following, for $0\leq j \leq r-1$:
  \[
  \begin{tikzcd}
  \X_r \ar[rd, bend right, "\id"'] \ar[r, "s_r"] & \X_{r+1} \ar[d, "d_{r+1}"]  \\
  \ar[ru, pos=0.7, phantom, "\gamma_r" description] & \X_r
  \end{tikzcd}
  \qquad
  \begin{tikzcd}
  \X_r \ar[d, "d_r"'] \ar[r, "s_j"] & \X_{r+1} \ar[d, "d_{r+1}"]  \\
  \X_{r-1}\ar[ru, phantom, "\gamma_j" description] \ar[r, "s_j"']& \X_r
  \end{tikzcd}
  \]
  For $\X=\nop(\catO)$ we define
  each $\gamma_r$ to be the identity. A simple calculation using
  Axiom~\ref{axQ:BM-67} shows that indeed $\dd_{r+1}\ss_r(p) = \id$, so
  that the definition makes sense. Similarly, for $j\le r-1$ we can put
  $\gamma_j$ equal to the identity, due to
  Axiom~\ref{axQ:BM-fibres-of-identities} and Axiom~\ref{axQ:BM-67}.

  We
  also have the other series of $2$-cells involving the top face and
  other face operators, for $0\leq j \leq r$:
  \begin{equation}\label{alpha}
  \begin{tikzcd}
  \X_{r+2} \ar[d, "d_{j}"'] \ar[r, "d_{r+2}"] & \X_{r+1} \ar[d, "d_j"]  \\
  \X_{r+1} \ar[r, "d_{r+1}"'] \ar[ru, Rightarrow, shorten <=14pt, shorten 
  >=14pt, "\alpha_{r,j}"]& \X_r \,.
  \end{tikzcd}
  \qquad 
  \begin{tikzcd}
  \X_{r+2} \ar[d, "d_{r+1}"'] \ar[r, "d_{r+2}"] & \X_{r+1} \ar[d, "d_{r+1}"]  \\
  \X_{r+1} \ar[r, "d_{r+1}"'] \ar[ru, Rightarrow, shorten <=14pt, shorten 
  >=14pt, "\beta_r"]& \X_r \,.
  \end{tikzcd}
  \end{equation}

  For $j=0$ and $p= (T_{r+2}\stackrel{f_{r+1}}{\to}
  T_{r}\stackrel{f_{r}}{\to} \ldots \stackrel{f_1}{\to} T_1)\in
  \nop(\catO)_{r+2}$, the operator $\dd_{r+1}\dd_{0}$ on $p$ returns the list of chains
  of morphisms
  \begin{equation}\label{listp}
      p_i = (h_{r+1}^{-1}(i)\xto{(f_{r-1})_i} \ldots \xto{(f_{2})_i} h_{1}^{-1}(i)), \  
      i\in |T_1| .
  \end{equation}
  On the other hand, $\dd_{r+2}(p)$ is equal to the list $q_i =
  (h_{r+1}^{-1}(i)\xto{(f_{r+1})_i} \ldots \xto{(f_{2})_i} h_{1}^{-1}(i)), \ i\in
  |T_1|$, and after application of $\dd_0$ we get exactly the list
  \eqref{listp}. Hence we can define $\alpha_0$ to be the identity. Similar
  calculations and Axiom~\ref{axQ:BM-67} allow us to define
  $\alpha_j$ to be the identity for all $0< j\leq r$.

  Now we turn to the definition of the last $2$-cell, $\beta$.  In general, we need a 
  monoidal natural transformation between strict monoidal functors (for $r\geq 0$)
  \[
  \begin{tikzcd}
  \X_{r+2} \ar[d, "d_{r+1}"'] \ar[r, "d_{r+2}"] & \X_{r+1} \ar[d, "d_{r+1}"]  \\
  \X_{r+1} \ar[r, "d_{r+1}"'] \ar[ru, Rightarrow, shorten <=14pt, shorten 
  >=14pt, "\beta_r"]& \X_r \,.
  \end{tikzcd}
  \]
  Since in our case all the objects are free symmetric strict monoidal groupoids,
  it is enough to define the component of $\beta$ on a $1$-element list. 

  Let us start with $r=1$, so as to be concerned with input from 
  $\nop(\catO)_3$.
  Let $p$ be 
  \begin{equation*} 
      \xymatrix@C = +1em@R = +1em{
          T      \ar[rr]^f \ar[dr]_h & & S \ar[dl]^g
          \\
          &R&
      }
  \end{equation*} 
  a monoidal generator in $\nop(\catO)_3$, so an element in $\catO_2$. We then have
  $\dd_2(p) = (T\stackrel{f}{\to} S)$ and $\dd_2\dd_2(p)$ is the list of
  fibres $f^{-1}(1),\ldots,f^{-1}(n)$, where $n= |S|$.
  On the other hand, $\dd_3(p)$ is the list of the fibre morphisms
  $f_1,\ldots,f_k$, where $k = |R|$, and $\dd_2\dd_3(p)$ is then the 
  flattened list of
  objects $(f_i)^{-1}(j), \ 1\le i \le k, j \in |g|^{-1}(i)$, where the order
  is lexicographic. According to 
  Axiom~\ref{axQ:BM-fibres-of-local-fibres}, these fibres match up under a 
  specific bijection also provided in the axiom:
  to each element in the list $(f^{-1}(1),\ldots,f^{-1}(n))$ 
  there corresponds an element in the   flattened list
  $((f_i)^{-1}(j)), \ 1\le i \le k, j \in |g|^{-1}(i)$. The individual 
  elements in the two lists are the same, thanks to the identification of
  the fibre $f^{-1}(j)$ with the fibre $f^{-1}_{|g|(j)}(j)$ 
  (Axiom~\ref{axQ:BM-fibres-of-local-fibres}). The bijection 
  takes element $j$ in the first list to 
  $\epsilon(j)$ in the $i$-th sublist of the nested list (with reference 
  to \eqref{eq:10}). More precisely, 
  the $p$-component of $\beta_1 : \dd_2 \dd_2 \Rightarrow 
  \dd_2\dd_3$ is precisely the $\pi$-part of the
  pita factorisation of $\norm g$.

  For $r=0$, the transformation
  \[
  \begin{tikzcd}
  \X_{2} \ar[d, "d_{1}"'] \ar[r, "d_{2}"] & \X_{1} \ar[d, "d_{1}"]  \\
  \X_{1} \ar[r, "d_{1}"'] \ar[ru, Rightarrow, shorten <=14pt, shorten 
  >=14pt, "\beta_0"]& \X_0 \,.
  \end{tikzcd}
  \]
  is constructed by applying ${\beta_1}$ to the monoidal generator $f =
  \id_T$. We then have the necessary permutation, this time provided by
  Axiom~\ref{axQ:BM-fibres-of-tau-maps}.

  In general, for $r \geq 2$ and a monoidal generator $p=
  (T_{r+2}\stackrel{f_{r+1}}{\to} T_{r+1}\stackrel{f_{r}}{\to} \ldots
  \stackrel{f_1}{\to} T_1)\in \nop(\catO)_{r+2}$, the operator 
  $\dd_{r+1}\dd_{r+1}$ on
  $p$ returns the list of chains of morphisms
  $$
  p_j = (e_{r}^{-1}(j)\xto{(f_{r+1})_j} \ldots \xto{(f_{3})_j} 
  e_{1}^{-1}(j)), \ j\in \norm{T_2} ,
  $$
  where
  \begin{equation*} 
      \xymatrix@C = +3em@R = +2em{
          T_{r+2}     \ar[r]^{f_{r+1}} \ar@/^-3.3ex/[drr]_{e_{r}} &  T_{r+1} 
          \ar[r]^{f_{r}}  \ar@/^-1.5ex/[dr]_{e_{r-1}} & \ldots  \ar[r]^{f_3} \ar[d] 
          &   T_3    \ar@/^+1.3ex/[dl]^{e_1:=f_2}
          \\
          &&T_2 &
      }
  \end{equation*} 
  The operator $\dd_{r+1}\dd_{r+2}$ on $p$ returns the list 
  $$
  (p_i)_j = (e_{r}^{-1}(j)\xto{((f_{r+1})_i)_j} \ldots 
  \xto{((f_{3})_i)_j} e_{1}^{-1}(j)), \ j\in |T_2|, \ i= |f_1|(j)\in |T_1|.
  $$
  Again by Axiom~\ref{axQ:BM-fibres-of-local-fibres} and functoriality there
  is a canonical isomorphism 
  $$
  (\beta_{r})_p : \dd_{r+1}\dd_{r+1}(p) \to
  \dd_{r+1}\dd_{r+2}(p)$$ 
  in $\nop(\catO)_{r}$, and it is given in bottom degree as the $\pi$-part of the pita
  factorisation of $\norm{f_1}$: once the entries in the direct $T_2$-indexed list and the
  flattened family of $(T_2)_j$-indexed lists have been matched up, the individual
  elements are identical by Axiom~\ref{axQ:BM-fibres-of-local-fibres}.
\end{blanko}

So far the construction gives a pre-simplicial structure on $\nop(\catO)$
for which the specified $2$-cells are invertible.
In each case, their components are certain permutations constructed 
using Axiom~\ref{axQ:BM-fibres-of-tau-maps} and
Axiom~\ref{axQ:BM-fibres-of-local-fibres}, and in each case this 
permutation arises from pita factorisation. 

We note that the operadic-nerve construction is functorial:
\begin{lemma}\label{operfun}
	Any operadic functor $p:\catP\to \catO$ induces a morphism of
	pre-simplicial symmetric monoidal groupoids
	$$
	\nop(p):\nop(\catP)\to \nop(\catO).
	$$
	Moreover, levelwise each $\nop(p)_r$ (for $r\ge 0$) is free, and in particular 
	is a discrete fibration of groupoids.
\end{lemma}

\begin{proof}
  The first statement follows because all the structure related to the
  top face and degeneracy operators and the beta cells were defined using
  chosen local terminal objects and fibres, and operadic functors are
  required to preserve such structure strictly. The components are
  discrete fibrations because they are obtained as $\SSS$ applied to a
  map of sets, and $\SSS$ preserves discrete fibrations.
\end{proof}

\section{Coherence of the operadic nerve}

The key feature of the operadic-nerve construction (and in the end an important
consequence of the axioms for operadic categories) is that the resulting structure is
coherent. That is, the permutation obtained from
Axiom~\ref{axQ:BM-fibres-of-local-fibres} is equal to a certain composite of
permutations and their inverses obtained from Axiom~\ref{axQ:BM-fibres-of-tau-maps}.
The coherence is the content of Theorem~\ref{NOP} below. Before coming to this
theorem, we need to spell out the specific kind of pseudo-simplicial groupoids we are
concerned with, and make explicit which are precisely its coherence equations. We
then illustrate the workings with the case of $\finset$.

\bigskip

Jardine~\cite{JARDINE1991103} lists all the 17 families of coherence equations. In our
case, only 3 out of 17 will be nontrivial; the others are strict. The main series of
equations is Jardine's Equation~(1.4.1). For all $i \leq j \leq k \leq r+1$ we have

\begin{equation}\label{beta-eq}
\begin{tikzcd}[column sep={44pt,between origins},row sep={35pt,between origins}]
\X_{r+3} \ar[dd, "d_i"'] \ar[rd, pos=0.7, "d_{j+1}"'] 
\ar[rr, "d_{k+2}"] && \X_{r+2}  \ar[rd, "d_{j+1}"] &  
\\
& \X_{r+2} \ar[dd, "d_{i}"] \ar[rr, "d_{k+1}"'] 
\ar[ru, Rightarrow, shorten <=8pt, shorten >=8pt]
&& \X_{r+1} \ar[dd, "d_{i}"] 
 \\
 \X_{r+2} \ar[rd, "d_j"'] \ar[ru, Rightarrow, shorten <=8pt, shorten >=8pt] & & &
 \\
 & \X_{r+1} \ar[rr, "d_k"'] 
 \ar[rruu, Rightarrow, shorten <=36pt, shorten >=36pt]&& \X_r
\end{tikzcd}
\ \ \
= 
\ \ \
\begin{tikzcd}[column sep={44pt,between origins},row sep={35pt,between origins}]
\X_{r+3} \ar[dd, "d_{i}"']  
\ar[rr, "d_{k+2}"] && \X_{r+2}  \ar[dd, "d_{i}"']\ar[rd, "d_{j+1}"] & 
\\
 & && \X_{r+1} \ar[dd, "d_{i}"]
 \\
 \X_{r+2} \ar[rd, "d_{j}"'] \ar[rr, "d_{k+1}"] 
 \ar[rruu, Rightarrow, shorten <=36pt, shorten >=36pt] & & \X_{r+1} \ar[rd, 
 "d_{j}"'] 
 \ar[ru, Rightarrow, shorten <=8pt, shorten >=8pt] &
 \\
 & \X_{r+1} \ar[rr, "d_{k}"'] 
 \ar[ru, Rightarrow, shorten <=8pt, shorten >=8pt]&& \X_r
\end{tikzcd}
\end{equation}
In our case, most of these are trivial. Only cases with high values of the indices are 
interesting, because only with two consecutive top face operators do we get a 
nontrivial cell, the $\beta$. With all indices equal $r+1$ we get the principal
equation, which in the case of $r=1$ reads
\begin{equation}\label{betacoherence1}
    \begin{tikzcd}[row sep={9ex,between origins}, column sep={4.6em,between origins}]
    d_{1} d_{1}d_{1}\ar[r, equal] \ar[d, "\beta_{0}d_{1}"'] 
    & d_{1}d_{1}d_{2} \ar[rr, "\beta_{0}d_{2}"]
    && d_{1} d_{2} d_{2} \ar[d, "d_{1}\beta_{1}"]
    \\
    d_{1}d_{2} d_{1} \ar[r, equal]
    & d_{1}d_{1}d_{3}  \ar[rr, "\beta_{0}d_{3}"']
    && d_{1}d_{2}d_{3} \,.
    \end{tikzcd}
\end{equation}
The equation one index down gives
\begin{equation}\label{betacoherence-new}
    \begin{tikzcd}[row sep={9ex,between origins}, column sep={6em,between origins}]
    d_{1} d_{1}d_{0}\ar[r, "\id"] \ar[d, "\beta_{0}d_{0}"'] 
    & d_{1}d_{0}d_{2} \ar[r, "\id"]
    & d_{0} d_{2} d_{2} \ar[d, "d_{0}\beta_{1}"]
    \\
    d_{1}d_{2} d_{0} \ar[r, "\id"']
    & d_{1}d_{0}d_{3}  \ar[r, "\id"']
    & d_{0}d_{2}d_{3} \,.
    \end{tikzcd}
\end{equation}
This is interesting too because it shows that $d_0(\beta_1)$ is completely determined
by $\beta_0$. Furthermore, $d_0$ in our case will be faithful, as it is $\SSS$ of a
map of sets, so already $\beta_1$ is determined by $\beta_0$. In fact, one can derive
from this equation that full coherence in all simplicial degrees is a consequence of
coherence in degree $3$ (the one we are mostly concerned with below). For the precise
statement, see Lemma~8.3 in the pita paper~\cite{Batanin-Kock-Weber:2512.22794}.

The second series (1.4.4 in the list of Jardine~\cite{JARDINE1991103}) of relevance concerns interplay with degeneracy 
operators. This time we list only the ones containing $\beta$-cells:
\begin{equation}\label{beta-eq-bis}
\begin{tikzcd}[column sep={44pt,between origins},row sep={35pt,between origins}]
X_{r+2}  \ar[rd, pos=0.7, "d_{r+1}"'] 
\ar[rr, "d_{r+2}"] && X_{r+1}  \ar[rd, "d_{r+1}"] &  
\\
& X_{r+1} \ar[rr, "d_{r+1}"'] 
\ar[ru, Rightarrow, shorten <=8pt, shorten >=8pt, "\beta_{r}"]
&& X_{r} 
 \\
 X_{r+1} \ar[uu, "s_{r+1}"]  \ar[rd, "\id"'] \ar[ru, phantom, "\commutes" description] & & &
 \\
 & X_{r+1} \ar[rr, "d_{r+1}"'] \ar[uu, "\id"']
 \ar[rruu, phantom, "\commutes" description]&& X_{r}  \ar[uu, "\id"']
\end{tikzcd}
\ \ \
= 
\ \ \
\begin{tikzcd}[column sep={44pt,between origins},row sep={35pt,between origins}]
X_{r+2}  
\ar[rr, "d_{r+2}"] && X_{r+1} \ar[rd, "d_{r+1}"] & 
\\
 & && X_{r} 
 \\
 X_{r+1} \ar[uu, "s_{r+1}"] \ar[rd, "\id"'] \ar[rr, "\id"] 
 \ar[rruu, phantom, "\commutes" description] & & X_{r+1} \ar[rd, 
 "d_{r+1}"'] \ar[uu, "\id"]
 \ar[ru, phantom, "\commutes" description] &
 \\
 & X_{r+1} \ar[rr, "d_{r+1}"'] 
 \ar[ru, phantom, "\commutes" description]&& X_r
 \ar[uu, "\id"']
\end{tikzcd}
\end{equation}
which says that $\beta$-cells are trivial on top-degenerate simplices.

The third series (1.4.2 in the list of Jardine~\cite{JARDINE1991103})
has lower indices, $i\leq r$, which is less interesting:
\begin{equation}\label{coherence-new-bis}
\begin{tikzcd}[column sep={44pt,between origins},row sep={35pt,between origins}]
X_{r+3}  \ar[rd, pos=0.7, "d_{r+2}"'] 
\ar[rr, "d_{r+3}"] && X_{r+2}  \ar[rd, "d_{r+2}"] &  
\\
& X_{r+2} \ar[rr, "d_{r+2}"'] 
\ar[ru, Rightarrow, shorten <=8pt, shorten >=8pt, "\beta_{r+1}"]
&& X_{r+1} 
 \\
 X_{r+2} \ar[uu, "s_{i}"]  \ar[rd, "d_{r+1}"'] \ar[ru, phantom, "\commutes" description] & & &
 \\
 & X_{r+1} \ar[rr, "d_{r+1}"'] \ar[uu, "s_{i}"']
 \ar[rruu, phantom, "\commutes" description]&& X_{r}  \ar[uu, "s_{i}"']
\end{tikzcd}
\ \ \
= 
\ \ \
\begin{tikzcd}[column sep={44pt,between origins},row sep={35pt,between origins}]
X_{r+3}  
\ar[rr, "d_{r+3}"] && X_{r+2} \ar[rd, "d_{r+2}"] & 
\\
 & && X_{r+1} 
 \\
 X_{r+2} \ar[uu, "s_{i}"] \ar[rd, "d_{r+1}"'] \ar[rr, "d_{r+2}"] 
 \ar[rruu, phantom, "\commutes" description] & & X_{r+1} \ar[rd, 
 "d_{r+1}"'] \ar[uu, "s_{i}"]
 \ar[ru, phantom, "\commutes" description] &
 \\
 & X_{r+1} \ar[rr, "d_{r+1}"'] 
 \ar[ru, Rightarrow, shorten <=8pt, shorten >=8pt, "\beta_n"]&& X_r
 \ar[uu, "s_{i}"']
\end{tikzcd}
\end{equation}
We list it for completeness, but
in our case, this equation can in fact be derived as a consequence of 
\eqref{betacoherence-new} (see \cite[Lemma~8.3]{Batanin-Kock-Weber:2512.22794}).

\begin{blanko}{Special case of $\finset$.}
  We analyse the special case of $\finset$, to illustrate the workings, and because
  this case will play a special role. Given a commutative triangle in $\finset$
  \begin{equation}\label{triangle}
  \begin{tikzcd}[column sep={30pt,between origins}]
  m \ar[rr, "f"] \ar[rd, "h"'] & \ar[d, phantom, pos=0.4, "\sigma" description] & n 
  \ar[ld, "g"]  \\
   & k &
  \end{tikzcd}
  \end{equation}
  then we have
  $$
  \dd_0(\sigma) = g  \qquad
  \dd_1(\sigma) = h \qquad 
  \dd_2(\sigma)= f \qquad
  \dd_3(\sigma) =: (f_1,\ldots,f_k) .
  $$
  Here the $k$-tuple $(f_1,\ldots,f_k)$ is the list of fibre maps of $f$ with 
  respect to $g$.
  We also put $(m_1,\ldots,m_k) := \dd_2 (h) = \dd_2\dd_1(\sigma)$ and
  $(n_1,\ldots,n_k) := \dd_2 (g) = \dd_2\dd_0(\sigma)$, and then it is an easy check
  with simplicial identities to see that 
  $\dd_1(f_i)= m_i$ and $\dd_0(f_i) = n_i$, so as to have
  \begin{equation}\label{fi}
  f_i : m_i \to n_i .
  \end{equation}

  We now spell out what the first few $\beta$-cells are. First
  $$
  (\beta_0)_f : \dd_1 \dd_1(f) \to \dd_1 \dd_2(f)
  $$
  On the domain we have $\dd_1 \dd_1(f)= \dd_1(m)$. This is the list of terminals of
  length $\norm{m}$. On the codomain we have $\dd_1 \dd_2(f) = \dd_1(f_1,\ldots,f_k)$.
  This is the flattened list of: the list of terminals of length $m_1$, followed by 
  the list of terminals of length $m_2$, and so on, ending with the list of 
  terminals of length $m_n$. Note that $m=\sum_{j=1}^n m_j$. Warning: the $m_j := f^{-1}(j)$ 
  (for $j\in n$) here are not
  the same as the $m_i:= h^{-1}(i)$ (for $i\in k$) from Equation~\eqref{fi}.
  
  The arrow $(\beta_0)_f$ is the bijection
  $$
  (\beta_0)_f : m \isopil \sum_{j=1}^n m_j
  $$
  arising as the $\pi$-factor in the pita factorisation of $f$.
  (Note that it is an arrow in the groupoid $\SSS\finset_{-1}$, not officially 
  an arrow in $\finset$. For $\finset$ this makes no difference, but in more
  general operadic categories there is a  difference.) 
  The $\eta$-part is also involved in the following
  calculations: it is ordinal sum of the unique maps $m_j \to 1$.
  
  Similarly,
  $
  (\beta_0)_g : n \isopil \sum_{i=1}^k n_i
  $
  and 
  $
  (\beta_0)_h : m \isopil \sum_{i=1}^k m_i .
  $
  (In these two cases we are back to talking about fibres indexed by $k$, so that $n_i := 
  g^{-1}(i)$ and $m_i := h^{-1}(i)$.)
  Note that the sum of the fibre maps,
  \begin{equation} \label{eta-as-fibre-maps}
  \sum_{i=1}^k f_i : \sum_{i=1}^k m_i \to \sum_{i=1}^k n_i
  \end{equation}
  is the connecting map $\eta(f/g)$ from Equation~\eqref{func-pita}, mediating between the 
  pita factorisation of $h$ and the pita factorisation of $g$.
   
  Let us now describe
  $$
  (\beta_1)_\sigma : \dd_2 \dd_2 (\sigma) \to \dd_2 \dd_3(\sigma).
  $$
  The domain is $\dd_2\dd_2(\sigma) = \dd_2(f)$, which is the list $(m_1,\ldots,m_n)$
  (where now $m_j = f^{-1}(j)$ for $j\in n$). The codomain is $\dd_2\dd_3(\sigma) = 
  \dd_2(f_1,\ldots,f_k)$, which is the flattened list
  $$
  \big( m_{11},\ldots,m_{1 n_1}, \ \ldots \ ,m_{k1},\ldots,m_{k n_k} \big)
  $$
  where
  $$
  m_{ij} = (f_i)^{-1}(j) \stackrel{\ref{axQ:BM-fibres-of-local-fibres}}{=} f^{-1}(j) 
  = m_j \subset m_i, \qquad \text{ for } j\in g^{-1}(i).
  $$
  The bijection $(\beta_1)_\sigma$ is
  $$
  (m_1,\ldots,m_n) \simeq \big( m_{11},\ldots,m_{1 n_1}, \ \ldots \ ,m_{k1},\ldots,m_{k n_k} \big)
  $$
  The entries in these two lists are the same (by 
  Axiom~\ref{axQ:BM-fibres-of-local-fibres}), but in the domain they are listed 
  according to the order of $n$, whereas in the codomain they are listed according to 
  their order down in $k$, and as a second criterion by their order within the fibre.
  The precise permutation is $(\beta_1)_\sigma = (\beta_0)_g = \pi_g$ (illustrating 
  \eqref{betacoherence-new}).
  
  In the coherence equation, it is rather $\dd_1$ of $(\beta_1)_\sigma$ that occurs.
  The effect of this $\dd_1$ is to split each of the list entries into a list of 
  terminals (and then flatten). So $\dd_1(\beta_1)_\sigma$ goes between two length-$m$ lists of 
  terminals, permuting the entries according to
  $$
  \sum_{j=1}^n m_j \isopil \sum_{i=1}^k \sum_{j=1}^{n_i} m_{ij} .
  $$
  Note that within each $m_{ij}$ identified with $m_j$, no permutation takes place,
  which is to say that the permutation is an $n$-shuffle: an element 
  $x\in m_j$ is sent to the same $x$ but now in the summand $m_{ij}$ (where $i=g(j)$).
  
  The map $\dd_1(\beta_1)_\sigma$ 
  can be described as the $\pi$-factor of the composite map 
  $$
  \sum_{j=1}^n m_j \to n \isopil \sum_{i=1}^k n_i .
  $$
  In precise terms,
  $$
  \dd_1(\beta_1)_\sigma = \pi( \pi_g \circ \sum f_i)= \pi( \pi_g \circ \eta_f) .
  $$
  
  Now let us look at the general coherence equation~\eqref{betacoherence1}
  \[
  \begin{tikzcd}[row sep={9ex,between origins}, column sep={4.6em,between origins}]
  \dd_{1} \dd_{1}\dd_{1}\ar[r, equal] \ar[d, "\beta_{0}\dd_{1}"'] 
  & \dd_{1}\dd_{1}\dd_{2} \ar[rr, "\beta_{0}\dd_{2}"]
  && \dd_{1} \dd_{2} \dd_{2} \ar[d, "\dd_{1}\beta_{1}"]
  \\
  \dd_{1}\dd_{2} \dd_{1} \ar[r, equal]
  & \dd_{1}\dd_{1}\dd_{3}  \ar[rr, "\beta_{0}\dd_{3}"']
  && \dd_{1}\dd_{2}\dd_{3} \,,
  \end{tikzcd}
  \]
  which in our case unpacks to
  \begin{equation}\label{coh-expanded}
  \begin{tikzcd}[row sep={12ex,between origins}, column sep={7.5em,between origins}]
  m \ar[d, "(\beta_{0})_h"'] 
   \ar[rr, "(\beta_{0})_f"]
  && {\displaystyle \sum_{j=1}^n f^{-1}(j)} \ar[d, "\dd_{1}((\beta_{1})_\sigma)"]
  \\
   {\displaystyle \sum_{i=1}^k h^{-1}(i)}  \ar[rr, "(\beta_{0})_{f_1,\ldots,f_k} = \sum_{i=1}^k (\beta_0)_{f_i}"']
  && {\displaystyle \sum_{i=1}^k \sum_{j=1}^{n_i} m_{ij}} \,.
  \end{tikzcd}
  \end{equation}

  It is far from obvious that this diagram commutes in general. Presumably one could
  check it by hand, writing out everything and keeping track carefully of all the
  maps. In a moment (Example~\ref{7-5-3} below) we shall give an example computation,
  both to illustrate how it works and to explain why we do not want to try the general
  case. Rather than trying to establish the general coherence by hand, further on in
  \ref{NOP} and \ref{F=N} we shall bypass such calculations by translating the
  problem into a setting (the pita nerve) where coherence holds for more general
  reasons.
  
  At the moment, taking for granted that the square commutes, let us first see how the
  pieces fit into the triangle \eqref{triangle} we are looking at:
  \begin{equation}\label{big-red-triangle-F}
  \begin{tikzcd}[column sep={10em,between origins},row sep={5.5em,between origins}]
  m \ar[d, "\pi_h"'] \ar[r, "\pi_f"] & {\displaystyle \sum_{j=1}^n m_j} \ar[d, 
  "\dd_1(\beta_1)_\sigma=\pi(\circ \pi_g\eta_f)"] \ar[r, "\eta_f"] & n
  \ar[d, "\pi_g"]  \\
  {\displaystyle \sum_{i=1}^k m_i} \ar[rd, bend right=20, "\eta_h"']\ar[r, "\pi_{(f_1,\ldots,f_k)}"] & 
  {\displaystyle \sum_{i=1}^k \sum_{j=1}^{n_i}} 
  m_{ij} \ar[r, "\eta(\pi_g \circ \eta_f)"] & {\displaystyle \sum_{i=1}^k n_i} \ar[ld, bend left=20, 
  "\eta_g"] 
  \\
  & k &
  \end{tikzcd}
  \end{equation}
  The upper right-hand square commutes since it is the pita factorisation of $\pi_g
  \circ \eta_f$, but the important point is that the middle vertical arrow is also
  $\dd_1(\beta_1)_\sigma$, as we already calculated earlier. With this interpretation, the upper
  left-hand square is the coherence equation~\eqref{betacoherence1}, as spelled out in 
  \eqref{coh-expanded}. The lower part of the diagram commutes by
  inspection: $\eta_h$ sends an element $x\in \sum_{i=1}^k m_i$ to the index of the
  summand it belongs to. That is also the description of the two other eta maps, on the
  right. Finally it just remains to observe that we have
  $$
  \pi_{(f_1,\ldots,f_k)} = \sum_{i=1}^k \pi_{f_i} ,
  $$
  it's a fibrewise map with respect to $i \in k$, so it does not change the property
  that $x$ is sent to $i$ if it belongs to the $i$-th fibre, which is commutativity of
  the lower part of the diagram.
\end{blanko}

\begin{blanko}{Example computation.}\label{7-5-3}
  Consider the triangle (in the operadic category $\finset$)
	\[
	\begin{tikzcd}[column sep={30pt,between origins}]
	7 \ar[rr, "f"] \ar[rd, "h"'] & \ar[d, phantom, pos=0.32, "\sigma" description] & 5 
	\ar[ld, "g"]  \\
	 & 3 &
	\end{tikzcd}
	\]
  with the maps given in picture form as
  
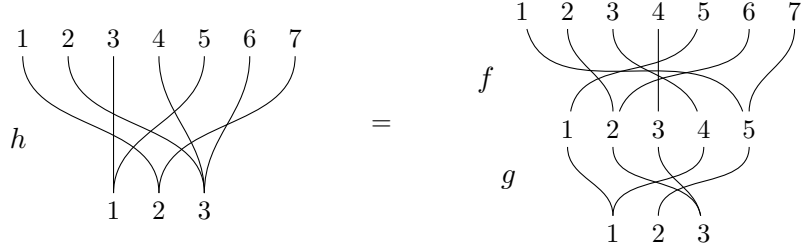
\[
  \begin{tikzpicture}[scale=0.6]
    \footnotesize
   
    \begin{scope}[shift={(1.0,-0.6)}]

    \begin{scope}[shift={(0,2.6)}]
      \draw (0.0,0) node (A1) {$1$};
      \draw (1.0,0) node (A2) {$2$};
      \draw (2.0,0) node (A3) {$3$};
      \draw (3.0,0) node (A4) {$4$};
      \draw (4.0,0) node (A5) {$5$};
      \draw (5.0,0) node (A6) {$6$};
      \draw (6.0,0) node (A7) {$7$};
    \end{scope}

    \draw (-0.1,0.4) node {\small $h$};
    
    \begin{scope}[shift={(0,-1.2)}]
      \draw (2.0,0) node (C1) {$1$};
      \draw (3.0,0) node (C2) {$2$};
      \draw (4.0,0) node (C3) {$3$};
    \end{scope}

    \draw (A1) to[out=-90, in=90] (C2);
    \draw (A2) to[out=-90, in=90] (C3);
    \draw (A3) to[out=-90, in=90] (C1);
    \draw (A4) to[out=-90, in=90] (C3);
    \draw (A5) to[out=-90, in=90] (C1);
    \draw (A6) to[out=-90, in=90] (C3);
    \draw (A7) to[out=-90, in=90] (C2);

  \end{scope}

  \draw (8.9,0.1) node {$=$};

  \begin{scope}[shift={(12,0)}]

    \begin{scope}[shift={(0,2.6)}]
      \draw (0.0,0) node (A1) {$1$};
      \draw (1.0,0) node (A2) {$2$};
      \draw (2.0,0) node (A3) {$3$};
      \draw (3.0,0) node (A4) {$4$};
      \draw (4.0,0) node (A5) {$5$};
      \draw (5.0,0) node (A6) {$6$};
      \draw (6.0,0) node (A7) {$7$};
    \end{scope}

    \begin{scope}[shift={(0,0)}]
      \draw (1.0,0) node (B1) {$1$};
      \draw (2.0,0) node (B2) {$2$};
      \draw (3.0,0) node (B3) {$3$};
      \draw (4.0,0) node (B4) {$4$};
      \draw (5.0,0) node (B5) {$5$};
    \end{scope}

    \draw (A1) to[out=-75, in=110] (B5);
    \draw (A2) to[out=-90, in=90] (B2);
    \draw (A3) to[out=-90, in=110] (B4);
    \draw (A4) to[out=-90, in=90] (B3);
    \draw (A5) to[out=-110, in=80] (B1);
    \draw (A6) to[out=-90, in=70] (B2);
    \draw (A7) to[out=-90, in=90] (B5);

    \begin{scope}[shift={(0,-2.3)}]
      \draw (2.0,0) node (C1) {$1$};
      \draw (3.0,0) node (C2) {$2$};
      \draw (4.0,0) node (C3) {$3$};
    \end{scope}

    \draw (B1) to[out=-90, in=90] (C1);
    \draw (B2) to[out=-90, in=100] (C3);
    \draw (B3) to[out=-90, in=100] (C3);
    \draw (B4) to[out=-90, in=90] (C1);
    \draw (B5) to[out=-90, in=90] (C2);
    
    \draw (-0.8,1.1) node {\small $f$};
    \draw (-0.3,-1.1) node {\small $g$};

  \end{scope}
  \end{tikzpicture}
  \]
  
  \noindent
  or in table form by
  
  \begin{footnotesize}
  $$
  f : \begin{array}{ccccccc} 1&2&3&4&5&6&7\\ 5&2&4&3&1&2&5\end{array}
  \qquad \qquad
  g : \begin{array}{ccccc} 1&2&3&4&5\\ 1&3&3&1&2\end{array}
  \qquad \qquad
  h : \begin{array}{ccccccc} 1&2&3&4&5&6&7\\ 2&3&1&3&1&3&2\end{array}
  $$
  \end{footnotesize}%
  We get the fibre maps (with respect to $g$)

  \begin{footnotesize}
  $$
  f_1 : \begin{array}{cc} 1&2\\ 2&1\end{array}
  \qquad \qquad
  f_2 : \begin{array}{cc} 1&2\\ 1&1\end{array}
  \qquad \qquad
  f_3 : \begin{array}{ccc} 1&2&3\\ 1&2&1\end{array}
  $$
  \end{footnotesize}%

  An explicit way of calculating pita factorisation is to look through the bottom row in 
  a table and number the entries by order of appearance. This means that the
  first appearance of `1' gets number 1, the second appearance of `1' gets number 2, 
  and so on, and then continuing with the appearances of `2', etc. (This is the 
  well-known process called {\em standardisation of words} in combinatorics; see for 
  example \cite{Duchamp-Hivert-Thibon:0105065}.) Write this 
  sequence of numbers as a middle row:

  \begin{footnotesize}
  $$
  f : \begin{array}{ccccccc} 1&2&3&4&5&6&7\\ 
  6&2&5&4&1&3&7
  \\
  5&2&4&3&1&2&5
  \end{array}
  \qquad \qquad
  g : \begin{array}{ccccc} 1&2&3&4&5
  \\ 
  1&4&5&2&3
  \\
  1&3&3&1&2\end{array}
  \qquad \qquad
  h : \begin{array}{ccccccc} 1&2&3&4&5&6&7\\
  3&5&1&6&2&7&4
  \\
  2&3&1&3&1&3&2\end{array}
  $$
  \end{footnotesize}%
  The upper half of the table is then the $\pi$-factor and the lower part is
  the $\eta$-factor. We also need the factorisations

  \begin{footnotesize}
  $$
  f_1 : \begin{array}{cc} 1&2\\2&1\\ 2&1\end{array}
  \qquad \qquad
  f_2 : \begin{array}{cc} 1&2\\1&2\\ 1&1\end{array}
  \qquad \qquad
  f_3 : \begin{array}{ccc} 1&2&3\\1&3&2\\ 1&2&1\end{array}
  $$
  \end{footnotesize}%

  Finally, to calculate $(\beta_1)_\sigma$ we need to compose

  \begin{footnotesize}
  $$
  \eta_f : 
  \begin{array}{ccccccc} 1&2&3&4&5&6&7\\ 
  1&2&2&3&4&5&5
  \end{array}
  $$
  \end{footnotesize}%
  
  with $\pi_g$ and then pita factor. Here is $\pi_g \circ \eta_f$:

  \begin{footnotesize}
  $$
  \pi_g \circ \eta_f : 
  \begin{array}{ccccccc} 1&2&3&4&5&6&7\\ 
  1&4&4&5&2&3&3
  \end{array}
  $$
  \end{footnotesize}%
  
  so the $\pi$-factor is the upper part of

  \begin{footnotesize}
  $$
  \pi_g \circ \eta_f : 
  \begin{array}{ccccccc} 1&2&3&4&5&6&7\\ 
    1&5&6&7&2&3&4\\ 
  1&4&4&5&2&3&3
  \end{array}
  $$
  \end{footnotesize}%
  
  We also need $\pi_{(f_1,f_2,f_3)}$:
 
  \begin{footnotesize}
  $$
  \pi_{(f_1,f_2,f_3)} : 
  \begin{array}{ccccccc} 1&2&3&4&5&6&7\\ 
  2&1&3&4&5&7&6
  \end{array}
  $$
  \end{footnotesize}%
  
  Now we verify the equation $\beta_\sigma \circ \pi_f = 
  \pi_{(f_1,f_2,f_3)} \circ \pi_h
  $ by calculating both sides, finding in both cases
 
  \begin{footnotesize}
  $$
  \begin{array}{ccccccc} 1&2&3&4&5&6&7\\ 
  3&5&2&7&1&6&4
  \end{array}
  $$
  \end{footnotesize}%

  thereby finalising the example computation, verifying the coherence equation in 
  this case.
\end{blanko}

We now begin the more formal work towards Theorem~\ref{NOP}.
The first step is an important
observation about operadic functors:

 \begin{lemma}\label{fibrelax}
  Let $p:\E\to \B$ be a strict morphism between pre-simplicial groupoids such that for
  each $n\ge 0$ the functor $p_n:\E_n\to \B_n$ is a discrete fibration. Then $\E$ is
  pseudo-simplicial provided $\B$ is pseudo-simplicial.
\end{lemma} 
  
\begin{proof}
  Since $p$ preserves coherence $2$-cells strictly, any composition of coherence
  cells in $\E$ is mapped by $p$ to precisely the same composition of 
  cells in $\B$. If
  in $\B$ two compositions of coherence cells coincide, then in $\E$ these two
  compositions must coincide as well, as otherwise this composition in $\B$ would have two
  different liftings to $\E$.
\end{proof}

\begin{theorem}\label{NOP}
  For any operadic category $\catO$, the pre-simplicial symmetric
  strict monoidal groupoid $\nop(\catO)$ is a coherent pseudo-simplicial
  symmetric strict monoidal groupoid whose only nontrivial $2$-cells
  are the $\beta_n$ constructed above.
\end{theorem}

\begin{proof}
  Lemma~\ref{operfun} applied to the cardinality functor $\catO \to \Fin$ induces a
  strict morphism of pre-simplicial groupoids $\nop(\catO)\to \nop(\Fin)$, which is
  level\-wise a discrete fibration. By Lemma~\ref{fibrelax} it is thus enough to
  establish coherence for the special case of the operadic nerve of $\Fin$. This check
  could possibly be done by hand, but it would be very cumbersome and somewhat
  unsatisfactory. Instead we will establish coherence by showing that $\nop(\Fin)$ is
  isomorphic to another pseudo-simplicial groupoid $\fnerve(\Fin)$, namely the
  so-called pita nerve of $\Fin$, which is coherent as a consequence of the main 
  theorem in~\cite{Batanin-Kock-Weber:2512.22794}, quoted as 
  Proposition~\ref{pitapaper-prop} below.
   
  The hard work has now been shifted to the task of establishing the
  isomorphism $\nop(\Fin)=\fnerve(\Fin)$, but this statement serves also to elucidate
  the structure of $\nop(\Fin)$. We state (and prove) this isomorphism as
  Proposition~\ref{F=N} below.
\end{proof}

Before going into the technical details of the arguments of the isomorphism, the following intuitive
explanation may be helpful. Looking at $\nop(\Fin)$ we see that all chains are `too
short': in simplicial degree $n$ they have length $n-1$, say $T_{n} \to \cdots \to
T_1$. We can fix that by appending a map to the terminal object, setting $T_0=1$. Next
we exploit the coincidence between the operadic category $\Fin$ we work with and the
indexing category for the symmetric-monoidal-groupoid monad (the same
skeletal $\Fin$). An $n$-simplex is a tuple of chains of length $n-1$, say a
$k$-tuple. After appending terminal objects, we can now take ordinal sum of these $k$
chains to obtain a single $n$-chain
$$
\sum_{i\in k} (T_n)_i \to  \sum_{i\in k} (T_{n-1})_i \to \cdots \to 
\sum_{i\in k} (T_1)_i \to k .
$$
This chain is locally order-preserving, meaning that all the composite maps down to
$k$ are order-preserving. Conversely, for every locally order-preserving chain in
$\Fin$ ending in $k$, we get a $k$-indexed tuple of $(n{-}1)$-chains given as the
fibres over each $i\in k$. From this viewpoint, we arrive quite naturally at what we
call the pita nerve. (In the companion paper~\cite{Batanin-Kock-Weber:2512.22794}, we
work in a more general setting and arrive at the pita nerve from general features
of pita factorisation.)
    
\begin{definition}[Cf.~\cite{Batanin-Kock-Weber:2512.22794}]
  We define the {\em pita nerve} $\fnerve(\Fin)$. The objects of $\fnerve(\Fin)_{n}$ are
  locally order-preserving chains
$$
T_{n}\stackrel{f_n}{\to} \ldots\stackrel{f_1}{\to} T_0
$$
in $\Fin$, meaning that all composites of morphism that end in $T_0$ are 
order-preserving morphisms. The morphisms in $\fnerve(\Fin)_n$
are fibrewise order-preserving 
commutative diagrams
\begin{equation}\label{fsplit}
    \xymatrix@C = +1em@R=0.5em{
     T_{n} \ar[dd]_{f_{n}} \ar[rr]^{\sigma_n}& &T_{n} \ar[dd]^{a_n} 
           \\
           & & \\
\vdots  \ar[dd]_{f_1}&  \vdots&\vdots \ar[dd]^{a_1} 
      \\ 
    & & \\
     T_0 \ar[rr]_{\sigma_0} &  &T_0   
}    
\end{equation}
in which all horizontal arrows are bijections (quasibijections, in the more general setting of
\cite{Batanin-Kock-Weber:2512.22794}).
To be fibrewise order-preserving means that each of the horizontal arrows 
is fibrewise order-preserving with respect to the fibres over any 
element
down in $T_0$.

The simplicial structure is given as follows: since objects in $\fnerve(\Fin)_n$ are
chains of length $n$, we immediately have all the usual simplicial operators at hand,
except the top face operators, which do not work naively: simply omitting the last
object $T_0$ will generally destroy the property of being locally order-preserving,
because there is no guarantee that the maps down to $T_1$ are order-preserving just
because they are so all the way down to $T_0$. However, the properties of pita
factorisation (in strictly factorisable operadic categories) ensures that such a
not-locally-order-preserving chain can be reflected back into being an locally
order-preserving chain. We proceed to describe this reflection.
(See~\cite{Batanin-Kock-Weber:2512.22794} for all details). There are also more
details in the proof of Proposition~\ref{F=N} below.)

Because of this intermediate step of reflection, the whole structure will
not be a strict simplicial category anymore, but only a pseudo-simplicial
object. The only nontrivial coherence $2$-cells are of the form $\beta$
and they are essentially given by pita factorisation.
In detail, for a locally order-preserving chain
$$
p : \qquad T_n \to \cdots \to T_0,
$$
applying first $d_{n-1}$ (the next-to-top face operator)  gives $T_n \to \cdots \to T_2 
\to T_0$,  and applying next 
$d_{n-1}$ (the top face operator) will omit $T_0$ and then refactor 
all the maps down to $T_2$. Functoriality of these factorisations supplies
the comparison maps so as to get a chain
$$
T_n' \to \cdots T_3' \stackrel{\eta_{f_2}} \longrightarrow T_2 \,.
$$
On the other hand, if starting with the same locally order-preserving 
chain $p$, we first apply the top face operator $d_n$ we get the 
$\eta$-part of the chain $T_n \to \cdots \to T_1$, which is a locally 
order-preserving chain 
$$
T_n' \to \cdots T_2' \stackrel{\eta_{f_1}}\longrightarrow T_1 .
$$
When we now apply the top face operator a second time, we first delete $T_1$
and then retain the $\eta$-part of the result. This is a locally 
order-preserving  $(n{-}2)$-chain 
$$
T_n'' \to \cdots T_3'' \stackrel{\eta_{f_2}} \longrightarrow T_2' \,.
$$
These two chains are not the same, but uniqueness of pita factorisation
provides a comparison: in the bottom degree it is precisely
$$
T_2 \stackrel{\pi_{f_2}}\longrightarrow T_2' .
$$
All the higher comparison maps are also $\pi$-factors of pita 
factorisation of the appropriate composite maps in the chain, and
altogether they form a fibrewise order-preserving bijection, as required
to constitute a morphism in the groupoid $\fnerve (\Fin)_{n-2}$.
\end{definition}

One of the main results of \cite{Batanin-Kock-Weber:2512.22794} is that
these $\beta$-cells are coherent.
We record this as a proposition:

\begin{proposition}[{\cite{Batanin-Kock-Weber:2512.22794}}]\label{pitapaper-prop}
  $\fnerve(\Fin)$ is a coherent pseudo-simplicial groupoid, in which the only 
  nontrivial $2$-cells are (for $n\geq 0$)
  \[
\begin{tikzcd}
\fnerve(\Fin)_{n+2} \ar[d, "d_{n+1}"'] \ar[r, "d_{n+2}"] & \fnerve(\Fin)_{n+1} \ar[d, "d_{n+1}"]  \\
\fnerve(\Fin)_{n+1} \ar[r, "d_{n+1}"'] \ar[ru, Rightarrow, shorten <=14pt, shorten 
>=14pt, "\beta_n"]& \fnerve(\Fin)_n \,.
\end{tikzcd}
\]
\end{proposition}

We are now ready for the main argument:

\begin{proposition}\label{F=N}
  The pre-simplicial groupoid $\nop(\Fin)$ is isomorphic to the
  pseudo\-simplicial groupoid $\fnerve(\Fin)$.
\end{proposition}

\begin{proof}  
  For $p\in \nop(\Fin)_n$ (that is, a chain $p = (T_n\to\cdots\to T_1)$ of length
  $n-1$ in $\Fin$), let $\tilde{p}$ be the chain of length $n$ given by
  $$
  T_n\to\ldots\to T_1\to 1 .
  $$ 
  This is clearly a locally order-preserving chain. A general object in $\nop(\Fin)_n$
  is a list $p_1,\ldots,p_k$ of $(n{-}1)$-chains of morphisms. We then take the
  ordinal sum
  $\tilde{p}_1+\ldots+\tilde{p}_k$ to obtain a locally order-preserving
  chain again, so it is an object of $\fnerve(\Fin)_n$.

  For the inverse construction, let $p =(T_{n}\xto{f_n}\ldots\xto{f_1}
  T_{0} = k)$ be a locally order-preserving chain which we write as
  a commutative diagram
  \begin{equation*} 
    \xymatrix@C = +3em@R = +2em{
      T_n     \ar[r]^{f_{n}} \ar@/^-3.3ex/[drr]_{h_{n}} &  T_{n-1}      
      \ar[r]^{f_{n-1}}  \ar@/^-1.5ex/[dr]_{h_{n-1}} & \ldots 
      \ar[r]^{f_2} \ar[d] &   T_1    \ar@/^+1.3ex/[dl]^{f_1=h_1}
      \\
      &&T_0 &
    }
  \end{equation*}
  We then construct a list of chains of morphisms between fibres: 
  $$
  p_i = (h_{n}^{-1}(i)\xto{(f_{n})_i} \ldots \xto{(f_{1})_i} h_{1}^{-1}(i)),
  \ i\in |T_0| = k .
  $$
  We have to check that for this list we have the equality $\tilde{p}_1
  +\ldots +\tilde{p}_k = p$. But this follows because $p$ was a
  locally order-preserving chain. This describes the effect on objects 
  in simplicial degree $n$.

  Next we check that this bijective correspondence works on morphisms as well, so as
  to produce an isomorphism of groupoids in simplicial degree $n$. The morphisms in
  the groupoid $\nop(\Fin)_n$ are permutations of the indexing set $k$. Under the
  construction just described, they produce morphisms between $n$-chains with the
  special property that in each level they are order-preserving on the fibres of the
  bottom permutation $T_0 \isopil T_0$ (which is just the original permutation
  $k\isopil k$). In other words, they are the $\pi$-part of the pita factorisations of
  the composite maps down to $T_0$. We claim that we get in this way also an
  isomorphism on automorphism groups. 
  
  In degree zero, it is clear that we have an identification between $
  \mathfrak{S}_k$ and $\Aut(T_0)$; in higher simplicial degree the claim follows
  because in both simplicial groupoids, all composites of $d_0$ are discrete
  opfibrations. In $\nop(\Fin)$ this is clear because $d_0$ is $\SSS$ of a map
  of sets, and in $\fnerve(\Fin)$ this is a consequence of uniqueness of pita
  factorisation (see \cite[Lemma~7.5]{Batanin-Kock-Weber:2512.22794} for details). We
  have now established an isomorphism of groupoids $\nop(\Fin)_n=
  \fnerve(\Fin)_n$ in each degree $n$.
    
  Next we need to match up the generating data, namely the simplicial operators and
  the $2$-cells $\beta$. It is clear that all the lower part of the simplicial
  structure matches up. The only nontrivial part is to check the top face operators
  and the $\beta$-cells. The top face operators and $\beta$ in $\fnerve(\Fin)$ are
  essentially given by pita factorisation, as explained above.
  The claim is that this square commutes:
  \[
  \begin{tikzcd}
  \nop(\Fin)_{n-1} \ar[d] & \nop(\Fin)_n \ar[l, "\dd_n"'] \ar[d]  \\
  \fnerve(\Fin)_{n-1} & \fnerve(\Fin)_n \ar[l, "\dd_n"]
  \end{tikzcd}
  \]
  
  For a $k$-tuple of $(n{-}1)$-chains $(p_i)_{i\in k} \in
  \nop(\Fin)_n$, if we first apply the top face operator $\dd_n$ up in
  $\nop(\Fin)$, we obtain the list of $(n{-}2)$-chains indexed and
  ordered according to the ordinal sum $\sum_{i\in k} (T_1)_i$. The
  isomorphism down to $\fnerve(\Fin)_{n-1}$ then actually effectuates
  this ordinal sum to produce a single $(n{-}1)$-chain (a locally order-preserving
  chain). The other way around in the diagram, we first take the
  ordinal sum over $k$ to get a locally order-preserving $n$-chain whose last
  arrow is $\sum_{i\in k} (T_1)_i \to k$. The top face operator down in
  $\fnerve(\Fin)$ then omits the map to $k$, and performs pita factorisation of
  all the maps to $\sum_{i\in k} (T_1)_i$. Both ways around in the diagram we thus
  arrive at a locally order-preserving $(n{-}1)$-chain whose last object is
  $\sum_{i\in k} (T_1)_i$ with order inherited from the ordinal sum. The
  chains before the last object are also the same, as they are just the
  fibres of the original chain over the elements in the various $(T_1)_i$.
  
  It remains to see that the beta $2$-cells match up, but we have already done the
  work, by describing the components of $\beta$ as being in both cases the
  $\pi$-factor of the pita factorisation of the map $f_2$.
\end{proof}

\section{Operadic categories as `simplicial sets'}
\label{sec:as-ssets}

Recall that by construction of the operadic nerve of an operadic category
	$\catO$ we have the equality
	$$
	\uuu\upperstar (\nop(\catO)) = \SSS\nerexs(\catO)  \,,
	$$
	and note that the right-hand side exists for any category with chosen local
	terminals.

We are going to identify conditions on an abstract $\X$ under which the 
operadic-category structure on $\catO$ can be reconstructed.
A main task is to identify the image of $\nop$.

\begin{definition}\label{shuffle}
  Let $\X$ be a pseudo-simplicial symmetric strict monoidal groupoid, such that
  $\uuu\upperstar \X$ is the local-terminals nerve of a category $\catO$ with chosen local 
  terminals. That is, we have $\uuu\upperstar \X = \SSS\nerexs(\catO)$. 
  Since each symmetric monoidal groupoid $\X_r$ is free on a set, in 
  particular the coherence $2$-cells are essentially just permutations, and the 
  following condition makes sense:
  
  We say that 
  $\X$ is {\em shuffle-like} when for each $(r+2)$-simplex $f\in \X_{r+2}$, the coherence cell
  $$
  \beta_{r}(f) :d_{r+1}d_{r+1}(f)\Rightarrow d_{r+1} d_{r+2}(f)
  $$
  is a shuffle. This means that it is of form
  $$
  \sum_{j=1}^n m_j \isopil m
  $$
  required to be order-preserving on each summand (so it is more precisely an 
  $n$-shuffle), as detailed in the next example.
\end{definition}

\begin{remark}
  Note that among the coherence equations we have $d_0 \beta_{r} (\sigma) =
  \beta_{r-1} (d_0 \sigma)$. Since $d_0$ is strict monoidal, it sends a list
  permutation to the same permutation applied to the list of images. So
  $\beta_{r}(\sigma)$ and $\beta_{r-1}(d_0(\sigma))$ have the same underlying
  bijection of sets. A consequence of this is that the bijection of sets underlying
  $\beta_r(\sigma)$ only depends on the last principal edge of $\sigma$. It also means
  that all coherence cells are determined by $\beta_0$.
\end{remark}

\begin{example}\label{ex:X3}
  A generator for $\X_3$ is a commutative triangle in $\catO$, say
  \[
  \begin{tikzcd}[column sep={30pt,between origins}]
  T \ar[rr, "f"] \ar[rd, "h"'] & \ar[d, phantom, pos=0.4, "\sigma" description] & S 
  \ar[ld, "g"]  \\
   & R &
  \end{tikzcd}
  \]
  Then we have
  $\dd_0(\sigma) = g$,  $\dd_1(\sigma) = h$,  $\dd_2(\sigma)= f$, and
  $\dd_3(\sigma) =: (f_1,\ldots,f_k)$.
  (Here the finite set  $k$ is thus defined as the set  that indexes the 
  list $\dd_3(\sigma)$.)
  We also have the lists $(T_1,\ldots,T_k) :=  \dd_2(h) = \dd_2 \dd_1 (\sigma) =
  \dd_1 \dd_3(\sigma) = \dd_1(f_1,\ldots,f_k)$ and
  $(S_1,\ldots,S_k) := \dd_2(g) = \dd_2 \dd_0(\sigma) = \dd_0 \dd_3(\sigma) = 
  \dd_0(f_1,\ldots,f_k)$,
  showing that as expected we have
  $$
  f_i : T_i \to S_i .
  $$
  We also need 
  $n$ := the set that indexes the list $\dd_2(f)$ and        
  $n_i$ := the set that indexes the list $\dd_2(f_i)$.
  Now $\beta_1(\sigma) : \dd_2 \dd_2 (\sigma) \to \dd_2 \dd_3(\sigma)$
  is the morphism in the groupoid $ \X_1 = \SSS\catO_0$ 
  $$
  (T_1,\ldots,T_n) \stackrel{\sim}\longrightarrow (T_{11},\ldots,T_{1n_1}, 
  T_{21},\ldots,T_{2n_2},\ \ldots,\ T_{k1},\ldots,T_{k n_k}) .
  $$
  Since the set $\catO_0$ is discrete, the morphism in the groupoid $\SSS\catO_0$
  are just permutations: the
  same objects appear in these two lists, but in different order, and 
  $\beta_1(\sigma)$ has underlying permutation of sets 
  $$
  n \isopil \sum_{i=1}^k n_i
  $$
  (which is also the underlying permutation of $\beta_0(g)$).
  To be an $n$-shuffle means that
  each sum inclusion followed by the inverse bijection
  $$
  n_i \subset \sum_{i=1}^k n_i \isopil n
  $$
  is order-preserving.
  
  To describe $\dd_1(\beta_1(\sigma))$, the main point is that the top face 
  operator
  $\dd_1$ is a symmetric monoidal functor, so it returns the 
  concatenation of the lists corresponding to each entry in the input list.
  The resulting bijection 
  $$
  \dd_1(\beta_1(\sigma)) : \sum_{j=1}^n m_j \isopil \sum_{i=1}^k \sum_{j=1}^{n_i} 
  m_{ij} 
  $$
  is thus the extension of
  $\beta_0(g) = \beta_1(\sigma)$ to a block permutation, meaning that
  $x\in m_j$ is sent to the same $x$ but now in the summand $m_{ij}$ (where $i=g(j)$).
  We record this as a lemma:
\end{example}

\begin{lemma}\label{shuffle-consequence}
  The coherence bijections $\beta_0(g)$ and 
  $\dd_1(\beta_1(\sigma))$ are related by the commutative square
\[
\begin{tikzcd}[column sep={8.5em,between origins},row sep={4.3em,between origins}]
\sum_{j=1}^n m_j \ar[d, 
"\dd_1(\beta_1(\sigma))"'] \ar[r] & n
\ar[d, "\beta_0(g)"]  \\
\sum_{i=1}^k \sum_{j=1}^{n_i} 
m_{ij} \ar[r] & \sum_{i=1}^k n_i 
\end{tikzcd}
\]
  where the horizontal maps are the `display maps' of the sums.
\end{lemma}

\begin{lemma}
  For any operadic category, the pseudo-simplicial groupoid $\nop(\catO)$ satisfies
  the shuffle condition.
\end{lemma}

\begin{proof}
  We already noted in \ref{beta=pi} that the $\beta$ are given as the $\pi$-factor of
  pita factorisation. It is thus fibrewise order-preserving, which translates 
  precisely into the shuffle condition.
\end{proof}

\begin{definition}
  A pseudo-simplicial groupoid is called {\em top-pseudo} when the only nontrivial
  coherence $2$-cells are the $\beta$-cells
		\begin{equation*}
			\beta_{r}:d_{r+1}d_{r+1}\Rightarrow d_{r+1} d_{r+2}\quad : \qquad \X_{r+2} \to 
            \X_r \qquad r\ge 0 .
		\end{equation*}
  That is, all simplicial identities that
  do not involve two consecutive top face operators are strict identities.
\end{definition}

We are now ready to give the inverse statement to Theorem~\ref{NOP}.

\begin{theorem}\label{OE}
	Let $\catO$ be a category with chosen local terminals. Suppose $\X$ is a
	top-pseudo-simplicial symmetric strict monoidal groupoid such that
	$$
	\uuu\upperstar (\X) =\SSS\nerexs(\catO) .
	$$ 
    If $\X$ is a shuffle-like then there is a 
    unique operadic-category structure on $\catO$ such that
	$$
    \X = \nop(\catO) .
    $$
\end{theorem}

\begin{proof}
  {\em Overview of the proof.} Since we aim to do the $\nop$-construction backwards, the idea is of course that the
  top face operators should define a fibre functor structure on $\catO$, but before we
  can verify the axioms, we need a cardinality functor, and this is not obvious,
  because it is not directly visible in $\X$. Surprisingly it can be extracted
  nevertheless by a rather intricate procedure: the functor's value on morphisms will
  be assembled by pita factorisation in $\finset$, with the $\eta$-part coming from
  the top degeneracy operators and the $\pi$-part coming from the $\beta$-cells.
  Functoriality of this subtle patchwork will be a consequence of coherence of the
  $\beta$-cell and the shuffle condition. Only once we have constructed the
  cardinality functor can we finally verify the axioms for an operadic category, which
  in the end will amount to simplicial identities.
    
  So our first task is to construct a functor $|-|:\catO\to \Fin$, which will be the
  cardinality functor for $\catO$, and show it is unique. To this end we construct a
  strict morphism of pseudo-simplicial symmetric monoidal groupoids $\cc: \X
  \to \nop(\Fin)$ which is levelwise free, and then the cardinality functor will
  appear as a restriction of $\cc$ to the generators of $\SSS\catO$. We then observe
  that our conditions determine this functor uniquely.
		
  {\em Definition of $\cc_0: \SSS\catO_{-1}\to \SSS\Fin_{-1}$.} Since
  $\cc_0$ has to be free, it amounts to giving $\catO_{-1} \to \Fin_{-1} = 1$,
  so uniqueness of $\cc_0$ is clear. Note that $\cc_0$ is a 
  symmetric monoidal functor. Its effect on a morphism is $\X_0$ is to take 
  underlying bijection (of the sets that index the lists (objects in $\X_0 = 
  \SSS\catO_{-1}$)).
  
  {\em Definition of $\cc_1: \SSS\catO_0\to \SSS\Fin_0$.} It is determined by
  a map $\catO_0 \to \Fin_0$, and this map has to send a monoidal 
  generator $T \in \X_1$ (that is, an object in $\catO$) to the
  finite set $m$ that indexes the list $d_1(T) \in \X_0 = \SSS\catO_{-1}$. In detail,
  we require commutativity of
  \begin{equation}\label{comp-with-d1}
  \begin{tikzcd}
  \SSS\catO_{-1} \ar[d, "\cc_0"'] & \catO_0 \ar[l, "d_1"'] \ar[d, "\cc_1"] \\
  \SSS\Fin_{-1} & \Fin_0 \ar[l, "\dd_1"],
  \end{tikzcd}
  \end{equation}
  but the finite set that indexes $d_1(T)$ is the same finite set as indexes
  $\cc_0 d_1(T)$, and this set has to be the set of fibres of $\id_m: m \to m$,
  which is $m$ itself. So if we want to satisfy \eqref{comp-with-d1}, this is the only possible 
  choice for $\cc_1$. Now we should check that this $\cc_1$ is also compatible with 
  $s_0$ and $d_0$. For $s_0$, the requirement is commutativity of
  \begin{equation}\label{comp-with-s0}
  \begin{tikzcd}
  \SSS\catO_{-1} \ar[d, "\cc_0"'] \ar[r, "s_0"] & \catO_0  \ar[d, "\cc_1"] \\
  \SSS\Fin_{-1} \ar[r, "\ss_0"'] & \Fin_0  .
  \end{tikzcd}
  \end{equation}
  But this is fine: for $U\in \catO_{-1}$ we have $\ss_0 \cc_0(U) = \ss_0(1) = 1 \in
  \Fin_0$, and going the other way around the square we get that $\cc_1(s_0(U))$ is the
  set that indexes the list $d_1 s_0(U)$. But that list is of length $1$ because $s_0$
  is a strict section to $d_1$. Compatibility with $d_0$ is automatic since $d_0$ is a
  free map.
  
  {\em Definition of $\cc_2:  \SSS\catO_1\to \SSS\Fin_1$.}
  To each morphism $f:T\to S$, seen as a monoidal generator in $\X_2$, we need to
  assign a morphism $\cc_2(f)$ in $\finset$. Since we need compatibility with the 
  face operators $d_0$ and $d_1$, we need the domain and codomain of $\cc_2(f)$ to
  be
  $$
  \cc_1(T) \stackrel{\cc_2(f)}\longrightarrow \cc_1(S) .
  $$
  Furthermore,
  compatibility with $d_2$ tells us that we need this map to have fibres
  $$
  (\cc_1(T_1),\ldots,\cc_1(T_j)) =: (m_1,\ldots,m_n),
  $$
  where as usual we write $(T_1,\ldots,T_n) := d_2(f)$, with $n:=\cc_1(T)$
  and $m_j := \cc_1(T_j)$.
  Knowing the fibres of a map in $\finset$ is precisely to know its $\eta$-factor,
  so we have
  $$
  \eta( \cc_2(f)) = \sum_{j=1}^n (m_j{\to}1) .
  $$
  By (strict) uniqueness of pita factorisation in $\finset$, the map $\cc_2(f)$
  is thus determined if just we can provide also the $\pi$-factor. But this is
  determined by the requirement that the simplicial map $\cc$ should preserve 
  $\beta$-cells: we need
  $$
  \cc_0(\beta_0(f)) = \beta_0( \cc_2(f)) ,
  $$
  and in $\nop\finset$ we know that $\beta_0$ is precisely the $\pi$-factor.
  We are thus forced to take
  $$
  \pi(\cc_2(f)) = \beta_0( \cc_2(f)) = \cc_0(\beta_0(f)) .
  $$
  That this choice actually works is because of the shuffle condition on $\X$:
  the underlying bijection of $\beta_0(f)$ is the identity after restriction to 
  each $m_j$, which is to say that it is fibrewise order-preserving.
  
  Finally we can define
  \begin{align*}
  \cc_2(f) := \eta(\cc_2(f)) \circ \pi(\cc_2(f)) &= \big( \sum_{j=1}^n (m_j{\to}1) 
  \big) \circ \cc_0(\beta_0(f)) \\
  &= (\sum s_1 \cc_1(d_2(f))) \circ 
  \cc_0(\beta_0(f)) .
  \end{align*}
  (In the last expression, the sum-sign without indices denotes the ordinal sum of a 
  list of maps in $\finset$.) The choices involved in this definition ensure already
  compatibility with $d_0$, $d_1$, $d_2$, and $\beta$, and it was the unique
  choice achieving this compatibility. It remains to check compatibility with 
  $s_0$ and $s_1$, corresponding to the cases $\id:T\to T$ and $T\to U$ ($U$ the 
  chosen local terminal of the component of $T$). In both cases we have $\beta_0 = 
  \id$ (by Equations~\eqref{beta-eq-bis} and \eqref{coherence-new-bis}), while the $\eta$-parts are $\sum(1{\to}1)$ and
  $m{\to}1$, respectively, as required.

  {\em Definition of $\cc_3:  \SSS\catO_2\to \SSS\Fin_2$.}
  A generator for $\X_3$ is a commutative triangle in $\catO$, say
  \[
  \begin{tikzcd}[column sep={30pt,between origins}]
  T \ar[rr, "f"] \ar[rd, "h"'] & \ar[d, phantom, pos=0.4, "\sigma" description] & S 
  \ar[ld, "g"]  \\
   & R &
  \end{tikzcd}
  \]
  Then as in Example~\ref{ex:X3} we have
  $
  d_0(\sigma) = g$, $d_1(\sigma) = h$,  $d_2(\sigma)= f$, and
  $d_3(\sigma) =: (f_1,\ldots,f_k)$,
  where $k:=\cc_1(R)$, because the indexing set of $d_3(\sigma)$ is the same as the
  indexing set of $d_0 d_0 d_3(\sigma)$ (since $d_0 d_0$ is strictly monoidal), which
  is equal to $d_1(d_0 d_0(\sigma)) = d_1(R)$. We also put $(T_1,\ldots,T_k) := d_2
  (h) = d_2\dd_1(\sigma)$ and $(S_1,\ldots,S_k) := d_2 (g) = d_2 d_0(\sigma)$, and
  then it is an easy check with simplicial identities to see that $d_1(f_i)= T_i$ and
  $d_0(f_i) = S_i$, so as to have
  $f_i : T_i \to S_i$.
  
  The free map $\cc_3$ 
  must send this triangle to a commutative triangle in $\Fin$. 
  Since we 
  require the simplicial-map square
  \begin{equation}\label{012}
  \begin{tikzcd}
  \catO_1 \ar[d, "\cc_2"'] & \catO_2 \ar[d, "\cc_3"] \ar[l, "d_i"']  \\
  \Fin_1 & \Fin_2 \ar[l, "\dd_i"]
  \end{tikzcd}
  \end{equation}
  to commute for each $i=0,1,2$, we are thus forced to take $\cc_3(\sigma)$ 
  to be the triangle with sides
          \[
  \begin{tikzcd}[column sep={40pt,between origins}]
  \cc_1(T) \ar[rr, "\cc_2(f)"] \ar[rd, "\cc_2(h)"'] & & \cc_1(S) 
  \ar[ld, "\cc_2(g)"]  \\
   & \cc_1(R) &
  \end{tikzcd} \qquad =: \qquad
      \begin{tikzcd}[column sep={40pt,between origins}]
  m \ar[rr, "\cc_2(f)"] \ar[rd, "\cc_2(h)"'] & & n
  \ar[ld, "\cc_2(g)"]  \\
   & k &
  \end{tikzcd}
  \]
  (the second picture of the triangle is only to fix the usage of the letter $m$, $n$, 
  $k$). But we need to check that this triangle commutes! To this end 
  we must use the definition of $\cc_2$ already made, so the triangle we want to 
  commute is the outline of
  \[
  \begin{tikzcd}[column sep={13em,between origins},row sep={7.5em,between origins}]
  m \ar[d, "\beta_0(h)"'] \ar[r, "\beta_0(f)"] & {\displaystyle \sum_{j=1}^n m_j} \ar[d, dotted,
  "d_1(\beta_1)_\sigma
  "] \ar[r, "\sum_{j=1}^n (m_j{\to}1)"] & n
  \ar[d, "\beta_0(g)"]  \\
  {\displaystyle \sum_{i=1}^k m_i }\ar[rd, bend right=20, "\sum_{i=1}^k (m_i{\to}1)"']\ar[r, dotted,
  "\beta_0({(f_1,\ldots,f_k)})", "\sum_{i=1}^k(\beta_0(f_i))"'] & {\displaystyle \sum_{i=1}^k \sum_{j=1}^{n_i}}
  m_{ij} \ar[r, dotted, "\sum_{i=1}^k \sum_{j=1}^{n_i} 
  (m_{ij}{\to}1)"] & {\displaystyle \sum_{i=1}^k n_i} \ar[ld, bend left=20, 
  "\sum_{i=1}^k(n_i{\to}1)"] 
  \\
  & k &
  \end{tikzcd}
  \]
  and our experience with beta cells for $\finset$ (see
  diagram~\eqref{big-red-triangle-F}) tells us how to fill this triangle, with the
  dotted maps. Now the upper-left square is precisely coherence of $\beta$ 
  (Equation~\eqref{betacoherence1}), and the
  upper-right square is the shuffle condition (or more precisely its consequence spelled
  out in~\ref{shuffle-consequence}). The bottom triangle clearly commutes,
  because both ways around it are the $k$-indexed sum of maps to the terminal.

  Having established commutativity of this triangle, we now have a valid candidate for
  $\cc_3(\sigma)$, and to meet the requirements \eqref{012} there is only one possible 
  choice, and we thus know that the assignment is compatible with $d_0$, $d_1$, and 
  $d_2$. Now we need to check compatibility with $d_3$. We have
  $$
  \cc_2(d_3(\sigma)) = (\cc_2(f_1),\ldots,\cc_2(f_k))
  $$
  which is precisely the horizontal dotted composite in the diagram. But this map
  mediates between the pita factorisation of $\cc_2(h)$ and the pita factorisation of
  $\cc_2(g)$, so it must be $\eta(f/g)$ (see~\eqref{func-pita}). But $\eta(f/g)$
  is precisely the result of $d_3$ on a commutative triangle in $\finset$ (cf.~our
  preliminary analysis of $\finset$ (see page \pageref{eta-as-fibre-maps})), so
  compatibility with $d_3$ is verified. Now that we know that $\cc_0$, $\cc_1$,
  $\cc_2$, $\cc_3$ are compatible with all face operators up to degree $3$ and also with
  $\beta_0$, we  get compatibility also with $\beta_1$ automatically, 
  since $\beta$ is determined
  by $\beta_0$ (by Equation~\eqref{betacoherence-new}) (see also
  \cite[Lemma~8.3]{Batanin-Kock-Weber:2512.22794})). Finally, compatibility with the
  degeneracy operators is easy, as it is just about the cases $g=\id$, $f=\id$, and
  $R$ chosen local terminal.

  {\em Higher $\cc_r$.} These cases are completely analogous, only with more 
  complicated notation, but it is not necessary to do them, because with $\cc_2$ and 
  $\cc_3$ we have already established a functor.

  So now we have constructed a functor 
  $$
  \norm{ \thg } : \catO \to \finset .
  $$
  (Its uniqueness so far is as a simplicial map. Potentially
  there were more conditions than necessary for the functor, but when we come to 
  checking the operadic-category axioms, these extra conditions will reappear.)

  {\em Fibre functors.}
  Now that we have a cardinality functor to refer to, we can finally define the fibre 
  functors $\fib_{S,j}: \catO/S \to \catO$: given $f: T \to S$,  
   for $j\in n := \norm{S}$ we define the $j$-th fibre to be the $j$-th
  entry in the list $d_2(f)$.
  Since by construction $\norm{S} = n$ is the set indexing the list $d_2(f)$, we have 
  just the right number of fibres. Since we are going to check that it is a functor and
  that it satisfies the axioms for fibre functors, it is worth giving a more formal 
  description of $\fib_{S,j}$.

  The key point is that all simplicial identities involving top face operators 
  against 
  other simplicial operators are strict. This means that the top face operators 
  constitute a {\em strict} natural transformation (the counit of upper decalage)
  $$
  d_{\text{top}} : \uuu\upperstar(\iii\kkk)\upperstar\uuu\upperstar \X \to \uuu\upperstar\X ,
  $$
  of $\tEM$-presheaves, which amounts to
  $$
  d_{\text{top}} : \SSS\DDD \catO \to \SSS\catO .
  $$
  We cannot yet cancel the $\SSS$ because although $d_{\text{top}}$ is 
  symmetric monoidal, it is not free.
  But we can precompose with the $\SSS$-unit to get 
  $$
  \DDD\catO \to \SSS\DDD \catO \to \SSS\catO.
  $$
  In degree $-1$, this is
  $$
  \catO_0 \to \SSS\catO_0 \to \SSS\catO_{-1}
  $$
  and now we can pick out $S$ by means of $1 \to \catO_0$
  and pull back along this map to get the slice category over $S$:
  $$
  1 \to \catO/S \to \SSS\catO/S \to \SSS\DDD\catO \to \SSS\catO.
  $$
  Since we have fixed $S$, when we land in $\SSS\catO$ we actually land in
  $\catO^n$ (where $n=\norm S$), and from here we can project onto the $j$-th factor to get finally
  $$
  \catO/S \to \SSS\catO/S \to \SSS\DDD\catO \to \SSS\catO \stackrel{\text{proj}}\longrightarrow \catO
  $$
  which is our fibre map
  $$
  \fib_{S,j} : \catO/S \to \catO
  $$

  Having constructed it formally like this, it is clear that $\fib_{S,j}$ is a
  functor (as required in \ref{data:operadic-Q3}).
  
  {\em Checking the axioms.}
  The formal construction of $\fib$ makes is clear that it preserves local terminal
  objects, which is
  Axiom~\ref{axQ:BM-fibres-of-identities}. Note that
  to say that $(\id_S)^{-1}(j)$ is local terminal is the strict simplicial identity
  \begin{equation}\label{dsS}
    \begin{tikzcd}
    \X_0 \ar[d, "s_0"'] & \X_1 \ar[l, "d_1"'] \ar[d, "s_0"]  \\
    \X_1 & \X_2  .\ar[l, "d_2"]
    \end{tikzcd}
  \end{equation}

  Axiom~\ref{axQ:BM-abs(lt)} is already checked by construction,
  because the cardinality functor was constructed to be compatible with the top 
  degeneracy operators (that is, it was constructed as a natural transformation of 
  $\tEM$-presheaves, not just simplicial sets).

  For Axiom~\ref{axQ:BM-67}, we need to check that the square
  \eqref{diag:fibres-and-cardinalities} of functors
  \[
  \begin{tikzcd}[column sep={6em,between origins}]
  \catO \ar[d, "\norm{ \thg }"'] & \catO_{\smash{/S}} \ar[l, "\fib_{S,j}"'] \ar[d, 
  "\norm{ \thg }_{/S}"]  \\
  \finset & \finset_{/\norm{S}} \ar[l, "\fib_{\norm S,j}"]
  \end{tikzcd}
  \]
  commutes. But this is precisely the fact that $\cc$ was constructed not only to be a 
  functor but more precisely also to commute with top face operators. That the square 
  commutes amounts to the simplicial-map identities
  \begin{equation}\label{check-slice}
  \begin{tikzcd}
      \SSS\catO_0 \ar[d, "\cc_1"'] & \catO_1 \ar[d, "\cc_2"] \ar[l, "d_2"']  \\
      \SSS\Fin_0 & \Fin_1 \ar[l, "\dd_2"]
      \end{tikzcd}
  \qquad
      \begin{tikzcd}
      \SSS\catO_1 \ar[d, "\cc_2"'] & \catO_2 \ar[d, "\cc_3"] \ar[l, "d_3"']  \\
      \SSS\Fin_1 & \Fin_2 \ar[l, "\dd_3"]
      \end{tikzcd}
  \end{equation}
  Note at this point that when we observed uniqueness of the cardinality functor,
  it was really uniqueness of a cardinality functor compatible with top degeneracy 
  operators and top face operators, so the uniqueness is on cardinality functors 
  satisfying Axioms~\ref{axQ:BM-abs(lt)} and~\ref{axQ:BM-67}.

  Axiom~\ref{axQ:BM-fibres-of-tau-maps} states that for $g: S \to U$ (the morphism to a chosen 
  local terminal) we have $g^{-1}(1) = S$, and for $T \stackrel{f} \to S \stackrel g 
  \to U$ we have $f_1 = f$. But these are just the fact that the top degeneracy 
  operators are {\em strict} sections to the top face operators. Specifically the two 
  conditions are the simplicial identities
  \begin{equation}\label{check-sd}
  \begin{tikzcd}[row sep={3.5ex,between origins}]
  \X_1 \ar[rd, "s_1"] \ar[dd, "\id"'] &   \\
   & \X_2 \ar[ld, "d_2"] \\
   \X_1 &
  \end{tikzcd}
  \qquad
  \begin{tikzcd}[row sep={3.5ex,between origins}]
  \X_2 \ar[rd, "s_2"] \ar[dd, "\id"'] &   \\
   & \X_3 \ar[ld, "d_3"] \\
   \X_2 &
  \end{tikzcd}
  \end{equation}

  Finally, Axiom~\ref{axQ:BM-fibres-of-local-fibres} is the `fibres of fibres' axiom.
  The first part states that for $T \stackrel f \to S \stackrel g \to R$ and with 
  $i\in\norm R$ and $j \in \norm{g}^{-1}(i)$ we have $(f_i)^{-1}(j) = f^{-1}(\epsilon 
  j)$. Turning around $\epsilon$ it is an instance of $\beta$: 
  the $j$-th fibre of of $f$ is equal to the $\beta_1(j)$-th fibre of $f_i$ (where $i = 
  g(j)$).
  But this is precisely the coherence square below left
  \begin{equation}\label{fib-fib}
  \begin{tikzcd}
  \X_{3} \ar[d, "d_{2}"'] \ar[r, "d_{3}"] & \X_{2} \ar[d, "d_{2}"]  \\
  \X_{2} \ar[r, "d_{2}"'] \ar[ru, Rightarrow, shorten <=14pt, shorten 
  >=14pt, "\beta_1"]& \X_1 \,.
  \end{tikzcd}
  \qquad
  \begin{tikzcd}
  \X_{4} \ar[d, "d_{3}"'] \ar[r, "d_{4}"] & \X_{3} \ar[d, "d_{3}"]  \\
  \X_{3} \ar[r, "d_{3}"'] \ar[ru, Rightarrow, shorten <=14pt, shorten 
  >=14pt, "\beta_2"]& \X_2 \,.
  \end{tikzcd}
  \end{equation}
  The second square is then precisely the second part of 
  Axiom~\ref{axQ:BM-fibres-of-local-fibres}.
  (Note that the mere existence of these beta cells is enough for the sake of verifying 
  Axiom~\ref{axQ:BM-fibres-of-local-fibres}. It is not needed here that the beta cells 
  are coherent. (Coherence was needed only to be able to reconstruct the cardinality functor.
  And conversely, if the cardinality functor is given, then coherence follows, from the 
  fact that it is a discrete fibration, as we argued in the proof of the Coherence 
  Theorem~\ref{NOP}).)

  We have finished the construction of the operadic-category structure on $\catO$,
  and established its uniqueness.
\end{proof}

\begin{blanko}{Punch line.}
  Each one of the axioms for operadic categories is now a simplicial identity!
  (Cf.~Equations~\eqref{dsS}--\eqref{fib-fib} in the proof.)
\end{blanko}

\begin{remark}
  The idea that the cardinality functor of an operadic category should somehow
  be implied data was explored also by Lack~\cite{Lack:1610.06282}, who
  axiomatised a notion of operadic category with trivial objects instead of
  chosen local terminals in a way that it was not necessary to mention any
  cardinality functor. But note that his notion is weaker than the original
  notion of operadic category. Theorem~\ref{OE} (and notably its proof) shows
  that the cardinality functor can be suppressed in an {\em isomorphic}
  axiomatisation (but that this data is instead encoded as coherence data).
\end{remark}

In order to stress that it is really exactly the notion of operadic categories we
are encoding, we have stated the theorem using the slightly awkward condition of 
an {\em equality} $\uuu\upperstar\X =
\SSS\nerexs\catO$. The theorem can be weakened by demanding only an isomorphism of
operadic categories here. This
amounts to characterising the essential image of $\nop$ instead of the strict image.
We state that as a corollary:

\begin{corollary}
  Let $\catO$ be a category with chosen local terminals. Let $\X$ be a
  top-pseudo-simplicial symmetric monoidal groupoid satisfying the shuffle condition.
  Suppose we are given a levelwise free isomorphism of split-augmented simplicial
  groupoids
  $$
  f:\SSS\nerexs(\catO)\simeq \uuu\upperstar(\X).
  $$ 
  Then there is a unique operadic category structure on $\catO$ and a levelwise free
  isomorphism of pseudo-simplicial symmetric monoidal groupoids
  $$
  \bar{f}:\nop\catO\to \X
  $$ such that $\uuu\upperstar(\bar{f}) = f$.
\end{corollary}
\begin{proof}
  Consider the image of the composite
  $$
  \nerexs(\catO)\to\SSS\nerexs(\catO)\xrightarrow{f} \uuu\upperstar (\X)
  $$
  where  the first map is the unit of the monad $\SSS$. Since $f$ is an
  isomorphism, this image is a split-augmented nerve of a category 
  $\catO'$ with
  chosen local terminals, together  with an isomorphism
  $\catO\xrightarrow{\phi}\catO'$ giving a factorisation of $f$ as
  $$
  \SSS\nerexs(\catO)\xrightarrow{\SSS\phi}\SSS\nerexs(\catO')\xrightarrow{\id} \uuu\upperstar (\X).
  $$     
  Theorem~\ref{OE} provides a unique operadic-category structure on $\catO'$,
  which we can now transport back along the isomorphism $\phi$
  (recall that the operadic-category
  structure cannot be transferred along equivalences of categories, but obviously can
  be transferred along isomorphisms).
\end{proof}

We now give a more functorial account of the results, aiming at describing the 
category $\OpCat$ as a strict pullback of categories. 
Recall that $\kat{TopPs}$
denotes the category whose objects are top-pseudo-simplicial symmetric monoidal 
groupoids, meaning pseudo-simplicial symmetric monoidal 
groupoids whose only non-trivial $2$-cells are the $\beta$-cells. 
We define
$\kat{TopPs}^{\operatorname{fr+sh}} \subset \kat{TopPs}$ to be the subcategory consisting
of top-pseudo-simplicial symmetric monoidal groupoids that are levelwise free,
and for which the $\beta$-cells are shuffle permutations as in \ref{shuffle}.
In both cases, the morphisms are symmetric strict monoidal 
simplicial maps (but in the following pullback the morphisms are actually forced to be
levelwise free, since those in the image of $\SSS$ are levelwise free).

\begin{theorem}\label{pbk-thm}
  The square
  \begin{equation}\label{Nlt-diagram}
\begin{tikzcd}
  \OpCat \ar[d, "\operatorname{forget}"'] \ar[r, "\nop"]& 
  \kat{TopPs}^{\operatorname{fr+sh}} \ar[d, "\uuu\upperstar"]\\
\Catlt \ar[r, "\SSS\nerexs"'] &
{[(\tEM)\op,\SMGp]}
\end{tikzcd}
\end{equation}
is a strict pullback of categories.
\end{theorem}

The theorem gives a characterisation of operadic categories:
Explicitly, the objects of this pullback are pairs $(\X,\CC)$
where $\X$ is a top-pseudo-simplicial symmetric monoidal groupoid, that is levelwise 
free and satisfies 
the shuffle condition, 
$\CC$ is a small category with chosen
local terminals, and we have the strict equality $\uuu\upperstar \X = \SSS
\ltnerve \CC$.
The morphisms from $(\X,\CC)$ to $(\X',\CC')$ are given by pairs $(F,f)$ where $F: \X \to
\X'$ is a strict simplicial map, and $f: \CC\to\CC'$ is a functor, required to be
compatible in the sense that this diagram commutes:
	\[
	\begin{tikzcd}[column sep ={8mm,between origins}]
	\uuu\upperstar \X \ar[d, "\uuu\upperstar(F)"'] & = & \SSS \ltnerve \CC \ar[d, 
	"\SSS\ltnerve(f)"] \\
	 \uuu\upperstar \X' & = & \SSS \ltnerve \CC'  .
	\end{tikzcd}
	\]

\begin{proof}
  Having already done most of the work, only a few easy checks remain.
  Commutativity of the square gives us the functor
  $$
  \OpCat \longrightarrow \Catlt \times_{\PrSh(\tEM)}  
  \kat{TopPs}^{\operatorname{fr+sh}} .
  $$
  Theorem~\ref{NOP} says that this functor is bijective on objects: an operadic
  category structure on a category-with-chosen-local-terminals $\CC$ can be constructed
  uniquely from a top-pseudo-simplicial symmetric monoidal groupoid $\X :
  \simplexcategoryop\to\SMGp$ with $\uuu\upperstar X = \SSS \nerexs \CC$ (in
  particular $\X$ is thus levelwise free) provided $\X$ satisfies the shuffle
  condition. Second we need to perform the same check on morphisms in the category,
  meaning that strict operadic functors can be reconstructed. In fact we have a
  functor in the opposite direction: given a strict simplicial map $F:\X \to \X'$
  that is symmetric monoidal and levelwise free, the restriction to monoidal
  generators provides a functor on the underlying operadic categories. This functor
  preserves chosen local terminals because $F$ is strictly compatible with the top
  degeneracy operators, and the functor preserves fibre structure because $F$ is
  strictly compatible with the top face operators. That the functor also preserves
  cardinality follows from the construction in Section~\ref{sec:as-ssets}.
  Having now made the functors explicit, it is clear that they are inverse to each
  other (strictly).
\end{proof}

\section{Operadic categories as $\DDD$-bialgebras}
\label{sec:D-bialg}

In this section we digress from the pure simplicial viewpoint to give another
interpretation of the axioms, of the flavour {\em an operadic-category structure on a
category is a sort of $\DDD$-bialgebra structure in the Kleisli category of $\SSS$}
(so the task is to say what precisely this means). This serves two purposes: firstly
it pinpoints quite precisely the balance between the strict and the pseudo witnessed
so far, and secondly it more closely relates back to
Garner--Kock--Weber~\cite{Garner-Kock-Weber:1812.01750}.

From \cite{Garner-Kock-Weber:1812.01750} we know that the category $\Catlt$ is
equivalent to the category of (strict) $\DDD$-coalgebras in $\Cat$. 
The underlying simplicial
sets are ($\Set$-valued) $\tKLEISLI$-presheaves rather than $\tEM$-presheaves.
Diagram~\eqref{Nlt-diagram} refines to
\[
\begin{tikzcd}
  \OpCat \ar[d, "\operatorname{forget}"'] & \\
\Catlt 
\ar[d, "\simeq" , "\operatorname{forget}"'] \ar[r, "\nerexs"] & \PrSh(\tEM) \ar[d, "\iii\upperstar"]  \\
\DDD\kat{-Coalg} \ar[d, "\operatorname{forget}"']\ar[r] & \PrSh(\tKLEISLI) \ar[d, "\kkk\upperstar"]  \\
\Cat \ar[r, "\strictnerve"'] & \PrSh(\simplexcategory)
\end{tikzcd}
\]

For formal reasons $\DDD$ induces a monad on $\DDD\kat{-Coalg}$, denoted $\widetilde\DDD$.
The unit for 
the monad $\widetilde\DDD$ at a $\DDD$-coalgebra $\CC$ has underlying 
(split) simplicial map 
$\tau: \CC \to \DDD \CC$, and the multiplication of
the monad $\widetilde\DDD$ at a $\DDD$-coalgebra $\CC$  has underlying 
(split) simplicial map 
$\DDD\epsilon: \DDD\DDD\CC\to\DDD\CC$.
(By Garner--Kock--Weber~\cite{Garner-Kock-Weber:1812.01750}, the 
$\widetilde\DDD$-algebras are unary operadic categories.)

So far we have been talking about simplicial sets, ordinary categories, 
and $\DDD$-coalgebras in $\Cat$, but the 
notions make sense also for simplicial groupoids,
categories internal to groupoids, and $\DDD$-coalgebras in $\Cat(\Grpd)$,
and this extension is needed to accommodate the 
symmetric-monoidal-groupoid 
monad $\SSS$. 
Henceforth $\DDD\kat{-Coalg}$ denotes the category of $\DDD$-coalgebras
in $\Cat(\Grpd)$.
At first we talk about {\em strict} simplicial groupoids -- the pseudo-ness
will only come in when we say (normal) pseudo-algebra (for the monad 
$\widetilde\DDD$ on the Kleisli category for $\SSS$ on $\DDD$-coalgebras 
in $\Cat(\Grpd)$).

The starting point is the observation that the comonad $(\DDD,\delta,\epsilon)$ and
the monad $(\SSS,\mu,\eta)$ enjoy a distributive law (on
$\Cat(\Grpd)$), in fact given by a strict equality: since $\SSS$ acts on simplicial
objects $\simplexcategory\op\to\Grpd$ by postcomposition with $\SSS:\Grpd\to\Grpd$
while $\DDD$ acts by precomposition with $\ttt
:\simplexcategory\op\to\simplexcategory\op$, the strict equality of endofunctors
$$
\DDD \SSS = \SSS \DDD 
$$
follows from strict associativity of
functor composition. 
This identity of endofunctors is readily
checked to be strictly compatible with the monad and comonad 
structures, which amounts to the equations
\begin{equation}\label{eq:DT-triangles}
\xymatrix @C=0.9ex {
   \DDD\SSS \X \ar@{=}[rr]\ar[rd]_{\epsilon_{\SSS \X}} && \SSS\DDD \X 
   \ar[ld]^{\SSS\epsilon_ \X}\\
   & \SSS \X
}
\qquad\qquad
\xymatrix @C=0.9ex {
& \DDD \X \ar[ld]_{\DDD\eta_{ \X}} \ar[rd]^{\eta_{\DDD \X}} & \\
\DDD\SSS \X \ar@{=}[rr] && \SSS\DDD \X .
}
\end{equation}

\begin{equation}\label{eq:DT-pentagons}
\xymatrix @C=0.6ex {
   &\DDD\SSS \X \ar@{=}[rr]\ar[ld]_{\delta_{\SSS \X}} && \SSS\DDD \X 
   \ar[rd]^{\SSS\delta_ \X}\\
   \DDD\DDD\SSS \X\ar@{=}[rr] &&\DDD\SSS\DDD  \X \ar@{=}[rr] && \SSS\DDD\DDD \X
}
\qquad\quad
\xymatrix @C=0.6ex {
 \DDD\SSS\SSS \X \ar[rd]_{\DDD \mu_{ \X}} \ar@{=}[rr] && \SSS\DDD\SSS \X \ar@{=}[rr] 
 && \SSS\SSS\DDD \X\ar[ld]^{\mu_{\DDD \X}} & \\
&\DDD\SSS \X \ar@{=}[rr] && \SSS\DDD \X .
}
\end{equation}

The distributive law ensures that the monad $\SSS$ lifts to a monad on
$\DDD\kat{-Coalg}$, denoted $\widetilde\SSS$. The distributive law also
lifts, strictly, so as to have a distributive law between two monads
$$
\widetilde\DDD \widetilde\SSS = \widetilde\SSS \widetilde\DDD .
$$
This in turn implies that $\widetilde\DDD$ lifts to the Kleisli category 
for $\widetilde\SSS$ (still denoted $\widetilde \DDD$).
We are interested in the restricted Kleisli category,
   namely the full subcategory spanned by the objects that are 
   $\SSS$ of (the nerve of) an ordinary category with chosen 
   local terminals. Recall that the symmetric-monoidal-groupoid
   monad $\SSS$ acts levelwise: when we write $\SSS\CC$ it is
   shorthand for $\SSS\strictnerve\CC$; this is {\em not} the
   same thing as the (nerve of the) free symmetric monoidal
   category on $\CC$.

\begin{theorem}\label{thm:bialg}
  The category $\OpCat$ of operadic categories (and strict operadic functors)
  is equivalent to the category of shuffle-type normal pseudo-$\widetilde\DDD$-algebras
  on the restricted Kleisli category for $\widetilde\SSS$.
\end{theorem}

This gives some conceptual content to the mixture of strictness and pseudo-ness. Since
we are talking strict $\DDD$-coalgebras, we have simplicial groupoids that are strict
on the $\tEM$-part (so strict except for the top face operators), and the fact that we
consider {\em normal} pseudo algebras is the strictness of the identities $d_n s_{n-1}
= \id$. The remaining pseudo-ness and the coherence is precisely what is means to be a
pseudo-algebra. The $2$-cells involved and the equations they satisfy are the standard
ones for pseudo-algebras for a monad. In the case of the decalage monad (on
coalgebras) they can be found in \cite{Batanin-Kock-Weber:2512.22794} where we spell
them out (in the oplax case) in connection with the statement that the pita nerve is
oplax simplicial. The shuffle condition is the same as in 
\ref{shuffle}.

In detail this structure unpacks to the following: An operadic category is a small
category $\CC$ equipped with strict simplicial maps
$$
\tau: \CC\to\DDD\CC \qquad\qquad
\fib: \DDD \CC \to \SSS \CC 
$$
satisfying the following axioms.

(1) \emph{$\tau$ is a $\DDD$-coalgebra structure}.  That is we have
\begin{equation}\label{ax:coalg}
\xymatrix @C=1.8ex {
   & \DDD \CC \ar[rd]^{\epsilon} & \\
   \CC \ar@{=}[rr] \ar[ru]^{\tau} && \CC }
   \qquad\qquad 
   \qquad \qquad
\xymatrix{
   \CC \ar[r]^{\tau}\ar[d]_{\tau} & \DDD \CC \ar[d]^{\DDD \tau} \\
   \DDD \CC \ar[r]_-{\delta} & \DDD \DDD \CC .}
\end{equation}

(2) \emph{$\fib$ is a coalgebra homomorphism from $(\DDD \CC, \delta)$ to 
$(\SSS \CC, \SSS \tau)$}:
\begin{equation}\label{ax:phi-coalgmap}
  \xymatrix @C=1.8ex {
   \DDD \CC \ar[d]_{\delta} \ar[rrr]^{\fib} &&& \SSS \CC \ar[d]^{\SSS \tau} \\
   \DDD\DDD \CC  \ar[rr]_-{\DDD\fib}  &&  \DDD\SSS \CC \ar@{=}[r] &\SSS \DDD \CC .
  }
\end{equation}

(3) \emph{$(\CC,\fib)$ is a shuffle-type normal pseudo-$\widetilde \DDD$-algebra in 
$\kat{Kl}_\SSS(\DDD\kat{-Coalg})$}:
This means that the compatibility with the unit is strict:
\begin{equation}\label{ax:tau-phi-eta}
  \xymatrix{
   & \DDD \CC \ar[rd]^{\fib} & \\
   \CC \ar[rr]_{\eta} \ar[ru]^{\tau} && \SSS \CC
  }
\end{equation}
whereas the compatibility with multiplication (the associative law) is
only up to a specified invertible $2$-cell $\beta$:

\begin{equation}\label{ax:fibres}
  \xymatrix @C=1.8ex {
   \DDD\DDD\CC \ar@{}[rrrdd]|{\beta}
   \ar[rr]^-{\DDD\fib}\ar[dd]_{\DDD\epsilon} && \DDD\SSS\CC \ar@{=}[r] & \SSS\DDD\CC \ar[d]^{\SSS\fib} \\
    &&& \SSS\SSS\CC  \ar[d]^{\mu} \\
    \DDD\CC \ar[rrr]_{\fib} &&& \SSS\CC .
}
\end{equation}
subject to the standard pseudo-algebra coherence constraints,
which are
precisely the ones we have already studied from the simplicial viewpoint  -- 
as well as the shuffle condition.
See \cite{Batanin-Kock-Weber:2512.22794} for details regarding the coherence 
conditions from the viewpoint of normal pseudo-algebras.

The algebra structure looks nicer in a more consistent Kleisli viewpoint: The fibre
functor $\fib: \DDD \CC \to \SSS \CC$ is naturally a morphism in the Kleisli category
of $\SSS$, and as such we denote it
$$\xymatrix{
\DDD\CC \ar[r]|-@{|}^-{\fib} & \CC
}$$
Now Axiom~\eqref{ax:phi-coalgmap} says that $\fib$ is a $\DDD$-coalgebra map:
\begin{equation}\label{ax:phi-coalgmap-Kl}
  \xymatrix  {
   \DDD \CC \ar[d]_{\delta} \ar[r]|-@{|}^-{\fib} & \CC \ar[d]^{\tau} \\
   \DDD\DDD \CC  \ar[r]|-@{|}_-{\DDD\fib}  &  \DDD \CC 
  }
\end{equation}
and Axioms~\eqref{ax:tau-phi-eta} and \eqref{ax:fibres} state that 
$\fib$ is a normal pseudo
$\DDD$-algebra structure on the coalgebra $\CC$:
Axiom~\eqref{ax:tau-phi-eta} states we have a strictly commutative 
diagram
\begin{equation}\label{ax:tau-phi-eta-Kl}
  \xymatrix{
   & \DDD \CC \ar[rd]|-@{|}^{\fib} & \\
   \CC \ar[rr]_{=} \ar[ru]^{\tau} && \CC
  }
\end{equation}
whereas Axiom~\eqref{ax:fibres} gives the $2$-cell 
\begin{equation}\label{ax:fibres-Kl}
  \xymatrix {
   \DDD\DDD\CC \ar@{}[rd]|{\beta}
   \ar[r]|-@{|}^-{\DDD\fib}\ar[d]_{\DDD\epsilon} & \DDD\CC  \ar[d]|-@{|}^-{\fib} \\
    \DDD\CC \ar[r]|-@{|}_-{\fib} & \CC
}
\end{equation}
subject to the coherence equations.

\section{$\catO$-operads as \ikeo maps}
\label{sec:operads}

The raison d'\^etre of operadic categories are the operads over them \cite{BMEH},
\cite{Batanin-Markl:1404.3886}, \cite{Batanin-Markl:2105.05198}
\cite{Batanin-Markl:1812.02935}. In this section we explain how the simplicial
viewpoint leads to a very elegant characterisation of operads (at least in cartesian
monoidal categories): We show that for an operadic category $\catO$ the category of
$\catO$-operads in $\Set$ is isomorphic to the category of free \ikeo morphisms over
$\nop(\catO)$. IKEO stands for {\em inner Kan and equivalence on objects}; it is an
important class of simplicial maps, studied in particular in the setting of
decomposition spaces~\cite{Galvez-Kock-Tonks:2409.03742}, since they induce algebra
homomorphisms at the level of incidence algebras~\cite{Galvez-Kock-Tonks:1512.07573}.

We recall from \cite{Batanin-Markl:1404.3886} the notion of operad for an operadic
category $\catO$ in a symmetric monoidal category $(V,\otimes,e)$. Let $E =
\{E(T)\}_{T \in \catO}$ be an $\Ob(\catO)$-indexed family of objects in $V$. For each
morphism $f:T\to S$, writing $T_j := f^{-1}(j)$ for each $j\in \norm{S}$, we define
\[
E(f) := \bigotimes_{j \in |S|} E({T_j}).
\]

\begin{definition}\label{Jarca_u_mne_prespala!}
	An \emph{operad over $\catO$} (or simply an \emph{$\catO$-operad}) is a
	family of objects  $\calP = \{\calP(T)\}_{T \in \catO}$ in $V$  together with units
	\[
	e\to \calP(U_c),\ c \in \pi_0(\catO),
	\]
	and for each $f:T\to S$ a composition law
	\[
	\gamma_f: \calP(f) \otimes \calP(S)\to \calP(T),
	\]
	satisfying the following axioms.
    
	\begin{itemize}
		\item[(i)] Let $T \stackrel f\to S \stackrel g\to R$ be morphisms in
		$\catO$ and $h=gf$.  Then the
		following diagram of composition  laws of $\calP$ combined with the
		coherence constraints in $V$  commutes:
		\[
		\xymatrix@C = 2em@R = .4em{
			\ar@/^2.5ex/[rrd]^(.56){\hskip .5em\bigotimes_{i}\gamma_{f_i} \otimes \id}
			\ar[dd]_(.45){\id \otimes \gamma_g}
			\displaystyle\bigotimes_{i
				\in |R|} 
			\calP(f_i) \otimes \calP(g) \otimes \calP(R) & &    
			\\  & &  \ar@/^/[dl]^{\gamma_h}
			\calP(h) \otimes \calP(R)\ .
			\\
			\ar[r]_(.77){\gamma_f}{\rule{0pt}{2em}}  
			\displaystyle\bigotimes_{i \in |R|}
			\calP(f_i) \otimes \calP(S) \cong   \calP(f) \otimes \calP(S)& \calP(T)&
		}
		\]
		\item[(ii)] The composite
		\[
		\xymatrix@1@C = +2em{ \calP(T) \ar[r]& \rule{0pt}{2em} \displaystyle
			\bigotimes_{i\in |T|} e \otimes \calP(T) \ar[r]&\rule{0pt}{2em}
			\displaystyle \bigotimes_{i\in |T|} \calP(U_{c_i})\otimes\calP(T)\ar[r]^(.56)= &
			\calP(\id_T)\otimes\calP(T)
			\ar[r]^(.65){\gamma_{\id}}&\calP(T) }
		\]
		is the identity for each $T \in \catO$.
		\item[(iii)] The composite 
		\[
		\xymatrix@1@C = +2.2em{ \calP(T)\otimes e \ \ar[r]&\ \calP(T)\otimes
			\calP(U_c) \ \ar[r]^= &\ \calP(!_T)\otimes \calP(U_c)\
			\ar[r]^{\hskip 1.8em \gamma_{!_T}}&\ \calP(T)},
		\]
		is the right unit coherence isomorphism in $V$ for each $T \in \catO$,    where  $!_T: T\to
		U_c$ is the unique morphism. 
	\end{itemize}
\end{definition}

A {\em morphism\/} $\calP' \to \calP$ of $\catO$-operads is a family of
morphisms \hbox{$\calP'(T) \to \calP(T)$} in $V$, one for each $T \in \catO$,
commuting with the composition laws and units. This defines the category
$\kat{Opd}(\catO)$ of $\catO$-operads.

\bigskip

We recall the following Definition~2.1 of \cite{Batanin-Markl:1404.3886} of a discrete
fibration between operadic categories:

\begin{definition}\label{psano_v_Myluzach}
  A strict operadic functor $p:\catP\to \catO$ is a {\em discrete operadic
  fibration} if
  \begin{itemize}
	  \item[(i)]
	  $p$ induces a bijection $\pi_0(\catP) \to
	  \pi_0(\catO)$ and
	  \item[(ii)]
	  for any morphism $f : T\to S$ in $\catO$ and any list of objects
	  $t_1,\ldots, t_{n}, s \in \catP$, where $n= |S|$, such that
	  \[
	  p(s) = S \mbox { and } p(t_j) = f^{-1}(j) \mbox { for  all} \ j
	  \in |S|,
	  \]
	  there exists a unique $\sigma : t\to s$ in $\catP$ such that
	  \[
	  p(\sigma) = f \mbox { and } t_j = \sigma^{-1}(j)\ \mbox { for all} \ j
	  \in |S|.
	  \]
  \end{itemize}
\end{definition}

The following characterisation of operads over an operadic category was
established in \cite{Batanin-Markl:1404.3886}, using a discrete operadic Grothendieck
construction.
\begin{proposition}[\cite{Batanin-Markl:1404.3886}]\label{opd=dof}
	There is an isomorphism between the category of $\catO$-operads in $\Set$  and the
	category of discrete operadic fibrations over $\catO$.
\end{proposition}

We shall see that discrete operadic fibrations, and therefore the notion of operad,
admit an elegant characterisation in terms of IKEO maps:

\begin{definition}
  A simplicial map $f: \Y \to \X$ (between simplicial sets or
  simplicial/ pseudo-simplicial groupoids) is called {\em IKEO}
  \cite{Galvez-Kock-Tonks:2409.03742} when it is inner Kan and an equivalence on 
  objects. 
  {\em  Inner Kan} means that for each $n$, the
  square
\begin{equation}\label{IKEO}
	\xymatrix@C = +4em{
		\Y_n 
		\ar[d]_{f_n} \ar[r]& \Y_1\times_{\Y_0}\cdots\times_{\Y_0} \Y_1 
		\ar[d]^{f_1\times_{f_0}\cdots\times_{f_0} f_1} 
		\\
		\X_n \ar[r] & \X_1\times_{\X_0}\cdots\times_{\X_0} \X_1
	}
\end{equation} 
is a homotopy pullback. Here the horizontal maps are the Segal maps, which take an
$n$-simplex to the $n$-tuple of its principal edges. {\em Equivalence on objects}
means an equivalence in degree 0.
\end{definition}

\begin{remark}
  One reason to be interested in IKEO maps is that in the case of decomposition
  spaces, they are the simplicial maps that induce algebra homomorphisms (covariantly)
  between incidence algebras~\cite{Galvez-Kock-Tonks:1512.07573}. For functors between
  categories, the inner-Kan condition is automatic, so in this case the IKEO condition
  reduces to asking for bijection on objects. Functoriality of the incidence-algebra
  construction in such maps goes back to
  Content--Lemay--Leroux~\cite{Content-Lemay-Leroux}.
\end{remark}

The notions of equivalence and homotopy pullback refer to the homotopy structure of
groupoids (or spaces), but in our situation they actually reduce to the strict notions
of isomorphism and strict pullback. Indeed, we are concerned with pseudo-simplicial symmetric 
strict monoidal groupoids 
$\X, \Y:\simplexcategoryop\to \SMGp$
that arise as operadic nerves, and their strict morphisms.
Strictness of morphisms implies that the squares \eqref{IKEO} commute strictly.
Recall furthermore that $\X$ and $\Y$ are levelwise free $\SSS$-algebras, and 
that strict morphisms are required to be levelwise free as well. That is,
they are $\SSS$ of morphisms of sets. In particular they are discrete fibrations of 
groupoid, and a commuting square is therefore a homotopy pullback if and only if it
is an ordinary pullback. Similarly, since $f_0$ is free, it is an equivalence if and 
only if it is an isomorphism.

\begin{proposition}\label{dof=ikeo}
  An operadic functor $\catP \to \catO$ is a discrete operadic fibration if and only
  if $\nop(\catP) \to \nop(\catO)$ is IKEO.
\end{proposition}

\begin{proof}
  Let $p:\catP\to \catO$ be a discrete operadic fibration and let $\nop(p)$ be the
  corresponding morphism of operadic nerves whose components are $\nop(p)_{n} = 
  \SSS(p_{n-1})$,
  $n\ge 0$. For $n=0$ we need to check that $\SSS(p_{-1})$ is an
  isomorphism, but this follows immediately from bijectivity of $\pi_0(p)$.
	
	Let now $n=2$. Form the corresponding commutative diagram:
	\begin{equation*}
		\xymatrix@C = +2em{
			\SSS\catP_1 \ar[d]_{\SSS(p_1)} \ar[r]^{}& 
			\SSS\catP_0\times_{\SSS\catP_{-1}} \SSS\catP_0
			\ar[d]^{\SSS(p_0)\times_{\SSS(p_{-1})}\SSS(p_0)} 
			\\
			\SSS\catO_1\ar[r]^{} & \SSS\catO_0\times_{\SSS\catO_{-1}} \SSS\catO_0
		}
	\end{equation*}   
	The top horizontal map in this diagram sends a generator $\sigma: t\to s$
	to a pair
	$$
	(d_2(\sigma),d_0(\sigma)) =  (\sigma^{-1}(1)\ldots\sigma^{-1}(n),s)
	$$ 
	where $n=|s|$. Similarly for the bottom horizontal map. The square above
	is a pullback if and only if for any $(t_1\ldots t_n,s)\in
	\SSS\catP_0\times_{\SSS\catP_{-1}} \catP_0$ with $n =|s|$ and $f:T\to S$
	in $\catO_1$ such that $p(s) =S$ with $p(t_j) =f^{-1}(j), \ j\in |S| =
	|s|$, there exists a unique $\sigma:t\to s$ with $p(\sigma) = f$ and $t_j
	= \sigma^{-1}(j)$. Since $p$ is a discrete operadic fibration, this
	condition is certainly satisfied.
	
	For the converse, observe that the pullback condition for this square is
	almost the condition on $p$ to be a discrete fibration except that we
	have an extra condition for $t_j$ to be in the same connected component
	as the $i$-th fibre of identity $\id:s\to s$. So a priori, there might
	exist an operadic functor which satisfies this pullback condition but is
	not a discrete fibration. Our task is to rule out that scenario. Indeed,
	given data
	$$
	(t_1\ldots t_n,s) \ , f:T\to S, \ p(s) = S, \ p(t_j)= f^{-1}(j) , \ j\in |S| = |s|
	$$
	for a lifting problem, we first observe that we have $p(\id_s^{-1})(j) =
	\id^{-1}_S(j), \ j\in |S|$, because operadic functors preserve fibres,
	identities and cardinalities. Let $t_j\xto{!^{t_j}}U_{c_j}$ be the unique
	map to a chosen local terminal. Then $p(!^{t_j})$ is the unique morphism
	$f^{-1}(j)\xto{} p(U_{c_j})$. On the other hand the commutative triangle
	\begin{equation*} 
		\xymatrix@C = +1em@R = +1em{
			T      \ar[rr]^f \ar[dr]_f & & S \ar[dl]^\id
			\\
			&S&
		}
	\end{equation*}  
	induces maps on fibres $!_j = f_j^{\id} :f^{-1}(j)\to \id_S^{-1}(j), j\in |S|$.
	Hence $p(!^{t_j})=!_j $ and, in particular, $p(U_{c_j}) = \id_S^{-1}(j) =
	p(\id_s^{-1}(j))$. But $p$ is a bijection on trivial objects and we
	conclude that $U_{c_j} = \id_s^{-1}(j)$. So the condition of the pullback
	was already hidden in the lifting data for an operadic fibration.
	
	The inner Kan condition for $n>2$ can now be obtained by a straightforward
	induction using the equivalence of the conditions for $n=2$.
\end{proof}

Combining Proposition~\ref{dof=ikeo} with Proposition~\ref{opd=dof},
we finally arrive at the promised simplicial characterisation of operads:

\begin{theorem}\label{IKEOoperad}
	For a operadic category $\catO$ there is an isomorphism of
	categories
	$$
	\kat{Opd}(\catO) \simeq \kat{TopPs}^{\operatorname{sh}}/_{\operatorname{IKEO}}\nop(\catO)
	\subset \kat{TopPs}^{\operatorname{sh}}/\nop(\catO)
	$$
	between the category of $\catO$-operads in $\Set$ and the
	full subcategory of category of the slice $\kat{TopPs}^{\operatorname{sh}}/\nop(\catO)$
	consisting of \ikeo maps.
\end{theorem}

\bigskip

\footnotesize

\noindent
Address: {Department of Mathematics, HSE University, 119048 Usacheva str., 6, Moscow, Russia}

\noindent
Email addresses: \texttt{bataninmichael@gmail.com, mbatanin@hse.ru} 

\medskip

\noindent
Address: {Universitat Aut\`onoma de Barcelona, University of Copenhagen, 
and Centre de Recerca Matem\`atica}

\noindent
Email address: \texttt{joachim.kock@uab.cat} 

\medskip

\noindent
Address: {Macquarie University, Sydney}

\noindent
Email address: \texttt{mark.weber.math@googlemail.com} 

\vfill

\hrule

  \noindent Grant info: We acknowledge support from grant
  No.~10.46540/3103-00099B from the Independent Research Fund Denmark,
  grant PID2024-158573NB-I00 (AEI/FEDER, UE) of Spain and grant
  2021-SGR-1015 of Catalonia, as well as the Severo Ochoa and Mar\'ia de
  Maeztu Program for Centers and Units of Excellence in R\&D grant number
  CEX2020-001084-M and the Danish National Research Foundation
  through the Copenhagen Centre for Geometry and Topology (DNRF151).

\end{document}